\documentclass[11pt]{amsart}
\usepackage{graphicx}
\usepackage{latexsym,amsmath,amsfonts,amscd, amsthm}
\usepackage{color}
\usepackage[colorlinks=true]{hyperref}
\hypersetup{urlcolor=blue, citecolor=red}
\usepackage{tikz}
\usetikzlibrary{arrows,backgrounds,snakes}
\usepackage{epsfig}
\usepackage{amssymb}
\usepackage{xcolor}
\usepackage{comment}
\usepackage{subfigure}
\usepackage{multirow}
\usepackage{mathrsfs}
\newtheoremstyle{plainNoItalics}{}{}{\normalfont}{}{\bfseries}{.}{ }{}
\theoremstyle{plain}
\newtheorem{thm}{Theorem}[section]
\newtheorem{cor}[thm]{Corollary}
\newtheorem{lem}[thm]{Lemma}
\newtheorem{defn}[thm]{Definition}

\newtheorem{prop}[thm]{Proposition}

\newcommand{\f}{\frac}

\newcommand{\beq}{\begin{equation}}
\newcommand{\eeq}{\end{equation}}
\newcommand{\beqa}{\begin{eqnarray}}
\newcommand{\eeqa}{\end{eqnarray}}
\newcommand{\bit}{\begin{itemize}}
\newcommand{\eit}{\end{itemize}}
\newcommand{\bedef}{\begin{defn}}
\newcommand{\edefn}{\end{defn}}
\newcommand{\bpro}{\begin{prop}}
\newcommand{\epro}{\end{prop}}

\newcommand{\RR}{\mathbb{R}}
\newcommand{\TT}{\mathbb{T}}
\newcommand{\NN}{\mathbb{N}}

\DeclareMathOperator\dD{d}

\newcommand{\mA}{\mathcal A}
\newcommand{\mB}{\mathcal B}
\newcommand{\mC}{\mathcal C}

\newcommand{\mF}{\mathcal{F}}

\newcommand{\mI}{\mathcal I}
\newcommand{\mJ}{\mathcal J}
\newcommand{\mP}{{\mathcal P}}

\newcommand{\mS}{\mathcal S}
\newcommand{\mT}{\mathcal T}

\newcommand{\xL}{{x_{j-{1}/{2}}}}
\newcommand{\xR}{{x_{j+{1}/{2}}}}

\newcommand{\iR}{{j+{1}/{2}}}

\newcommand{\testR}{{\varphi}}   

\newcommand{\pf}{\partial f}

\newcommand{\pt}{\partial t}
\newcommand{\px}{\partial x}

\newcommand{\pv}{\partial v}

\newcommand\ds{ \displaystyle }

\newcommand{\dwdx}{\dD\mu_t}
\newcommand{\dw}{\omega(t)\dD v}
\newcommand{\PP}{\mathscr P}
\newcommand{\HH}{H}
\newcommand{\CC}{\mathscr C}
\usepackage{bbm}

\email{francis.filbet@math.univ-toulouse.fr}
\email{marianne.bessemoulin@univ-nantes.fr}

\title[Discontinuous Galerkin/Hermite Spectral methods for Vlasov-Poisson]{On the convergence of discontinuous Galerkin/Hermite spectral methods for the  Vlasov-Poisson system} 

\author{Marianne Bessemoulin-Chatard and Francis Filbet}



\begin{document}

\maketitle

\begin{abstract}
We prove the convergence  of  discontinuous Galerkin approximations
for the Vlasov-Poisson system written as an hyperbolic system using
Hermite polynomials in velocity. To obtain stability properties,
we introduce a suitable weighted $L^2$  space, with a time dependent weight,
and first prove global stability for the weighted $L^2$ norm and
propagation of regularity. Then we prove error estimates between the
numerical solution and the smooth solution to the Vlasov-Poisson system.
\end{abstract}

\tableofcontents

\section{Introduction}
\setcounter{equation}{0}
\label{sec:1}
We consider a noncollisional plasma of charged particles (electrons
and ions). For simplicity, we assume that the properties of the plasma
are one dimensional and   we take into account only the electrostatic
forces, thus neglecting the electromagnetic effects. We denote by
$f=f(t,x,v)$ the electron distribution function and by
$E(t,x)$  the electrostatic field. The
Vlasov-Poisson equations of the plasma in dimensionless variables can
be rewritten as,
\beq
\label{vlasov} 
\left\{
\begin{array}{l}
\ds\f{\pf}{\pt}\,+\, v\,\f{\pf}{\px} \,+\,E\,\f{\pf}{\pv} \,\,=\,\, 0\,,
\\[1.1em]
  \ds\f{\partial E}{\px}=\rho- \rho_0\,,
  \\[1.1em]
  \ds f(t=0) = f_0\,, 
\end{array}\right.
\eeq
where the density $\rho$ is given by
$$
\rho(t,x)\,=\,\int_\RR f(t,x,v)\,\dD v\,, \quad t\geq 0, \, x\in \TT.
$$
To ensure the well-posedness of the Poisson problem, we add the compatibility (or normalizing) condition
\begin{equation}\label{bc00}
\int_\TT \rho(t,x) \dD x =\int_\TT\int_{\RR} f(t,x,v)\dD v \dD x
\,=\,{\rm mes}(\TT)\,\rho_0,
\quad \forall \,\, t\geq 0,
\end{equation}
which is the condition for total charge neutrality. Let us notice that \eqref{bc00} express that the total charge of the system is preserved in time.


There is a wide variety of techniques  to discretize  the
Vlasov-Poisson system. For instance, particle methods (PIC) consist
in approximating the distribution function by a finite number of Dirac masses~\cite{Birdsall1985}. They allow to obtain satisfying results with a
\textcolor{black}{small} number of discrete particles, hence these methods are very popular
in the community of computational plasma physics, but a well-known drawback of this approach is the inherent numerical noise which only decreases in
$1/\sqrt{N}$ when the number of discrete particles $N$ increases, preventing
from getting an accurate description of the distribution function for
some specific applications. To overcome this difficulty, Eulerian
solvers have been applied. These methods discretize the Vlasov equation on a mesh of
the phase space \cite{Filbet2003,duclous,sonnen,filbet2006}. Among them, we can
mention finite volume methods \cite{Filbet2001} which are a simple and
inexpensive option, but in general low order. Fourier-Fourier
transform schemes \cite{Klimas1994} are based on a Fast Fourier
Transform of the distribution function in phase space, but suffer from
Gibbs \textcolor{black}{phenomenon} if other than periodic conditions are
considered. Standard finite element methods \cite{Zaki1988a,Zaki1988b}
have also been applied, but may present numerical oscillations when
approximating the Vlasov equation. Later, semi-Lagrangian schemes have
also been proposed \cite{Sonnendrucker1999}, consisting in computing
the distribution function at each grid point by following the
characteristic curves backward. Despite these schemes can achieve high
order allowing also for large time steps, they require high order
interpolation to compute the origin of the characteristics, destroying
the local character of the reconstruction. Many improvement have been
proposed  and studied to make this approach more efficient \cite{besse,
  charles,chang}. Finally, spectral Galerkin and spectral collocation methods for the asymmetric
weighted Fourier-Hermite discretization have been proposed in \cite{engelmann1963, LeBourdiec2006, manzini2016}. In
\cite{Camporeale2016}, \textcolor{black}{the} authors study a time implicit method allowing \textcolor{black}{the} exact conservation of charge, momentum and energy, and highlight that for some test cases, this scheme can be significantly more accurate than the PIC method. 

In the present article, we focus on a class of Eulerian methods based
on Hermite polynomials in the velocity variable, where the
Vlasov-Poisson system \eqref{vlasov}  is written as an hyperbolic
system. This idea of using Galerkin methods with a small finite set of
orthogonal polynomials rather than discretizing the distribution
function in velocity space goes back to the 60's \cite{Armstrong1967,
  Joyce1971}. More recently, the merit to use rescaled orthogonal
basis like the so-called scaled Hermite basis has been shown
\cite{engelmann1963, Holloway1996,holloway2,Schumer1998,Tang1993}. In
\cite{Holloway1996}, Holloway formalized two possible approaches. The
first one, called symmetrically-weighted, is based on standard Hermite
functions as the basis in velocity and as test functions in the
Galerkin method. It appears that this symmetrically weighted method
cannot simultaneously conserve \textcolor{black}{the} mass and \textcolor{black}{the} momentum. It makes up for this
deficiency by correctly conserving the $L^2$ norm of the distribution
function, ensuring the stability of the method. In the second
approach, called asymmetrically-weighted, another set of test
functions is used, leading to the simultaneous conservation of mass,
momentum and total energy since the infinite hyperbolic system corresponds to
the one satisfied by the moments of the distribution in the velocity
space. However, this approach  does not conserve the $L^2$ norm of the
distribution function and is then not numerically stable. Recently in
\cite{BCF2021}, we provide a stability analysis of the asymmetric
Hermite method in a weighted $L^2$ space for the Vlasov-Poisson
system. The main idea is to introduce a scaling function $t\mapsto
\alpha(t)$ which is well adapted to the variation of the distribution
function with respect to time.  The aim of this work is to present a
convergence analysis with error estimates  based on the asymmetric
weighted Hermite method with a discontinuous Galerkin method for the
space discretization. It is worth to mention that the convergence of \textcolor{black}{the} symmetric weighted
Fourier-Hermite method has been already studied in \cite{Manzini2017}
where the standard $L^2$ framework is well adapted. In
\cite{Kormann2021}, \textcolor{black}{the} authors study \textcolor{black}{the} conservation and $L^2$ stability
properties of a generalized Hermite-Fourier semi-discretization,
including as special cases the symmetric  and asymmetric weighted
approaches. Concerning discontinuous Galerkin methods, they are
similar to finite elements methods but use discontinuous polynomials
and are particularly well-adapted to handling complicated boundaries
which may arise in many realistic applications. Due to their local
construction, this type of methods provides good local conservation
properties without sacrificing the order of accuracy. They were
already used for the Vlasov-Poisson system in
\cite{Heath2012,Cheng2013}. Optimal error estimates and study of the
conservation properties of a family of semi-discrete DG schemes for
the Vlasov-Poisson system with periodic boundary conditions have been
proved for the one \cite{Ayuso2011} and multi-dimensional
\cite{Ayuso2012} cases. In all these works, the discontinuous Galerkin method is employed using a phase space mesh.

Here, we adopt this approach only in physical space, as in
\cite{Filbet2020}, with a Hermite approximation in the velocity
variable. In \cite{Filbet2020}, such schemes with discontinuous
Galerkin spatial discretization are designed in such a way that \textcolor{black}{the}
conservation of mass, momentum and total energy is rigorously
provable. In the next section, we introduce the formulation of the
Vlasov equation using the Hermite basis in velocity and a class of
spatial discretizations based on discontinuous Galerkin
approximations. Then we present our main result on error estimates
(Theorem \ref{th:main}). In Section \ref{sec:3}, we prove some
preliminary results on approximation theory based on Spectral accuracy
of Hermite spectral methods and remind some basic results on
interpolation error for discontinuous Galerkin method. Then in Section
\ref{sec:4}, we prove an error estimate between the semi-discrete
solution (time is continuous) and the exact smooth solution of the
Vlasov-Poisson system. Finally in
Section \ref{sec:5} we present numerical results in order to
illustrate the order of convergence and the stability of our approach.

%
%

\section{Discontinuous Galerkin/Hermite spectral methods}
\setcounter{equation}{0}
\label{sec:2}
In this section, we present the Discontinuous Galerkin/Hermite
spectral method. On the one hand, we focus on the velocity
discretization by expanding the distribution function $f$ using
Hermite polynomials. Then, we treat the space discretization using a
discontinuous Galerkin method \cite{Filbet2020,BCF2021}.

\subsection{Hermite spectral form}
For a given scaling positive function $t\mapsto \alpha(t)$ which will be determined
later, we define the weight as
\begin{equation}
  \label{eq:def_omega}
\omega(t,v)\;:=\,\sqrt{2\pi} \,\exp\left(\frac{\alpha^2(t)\,|v|^2}{2}\right),
\end{equation}
 and the associated weighted  $L^2$ space 
\begin{equation*}
L^2(\dw):=\left\{g : \RR\to\RR : \int_{\RR} |g(v)|^2\,\omega(t,v)\dD v <+\infty\right\},
\end{equation*}
with $\langle\cdot,\cdot\rangle_{L^2(\dw)}$ the inner
product and $\|\cdot\|_{L^2(\dw)}$ the corresponding norm. As in \cite{BCF2021}, we choose the following basis of normalized scaled time dependent asymmetrically weighted Hermite functions:
\beq
\label{hbasisf}
\Psi_n(t,v)\,=\,\alpha(t)\,{H}_n\left(\alpha(t)v\right)\,\f{e^{-(\alpha(t)v)^2/2}}{\sqrt{2\pi}}\,,
\eeq
where $\alpha$ is a scaling function depending on time and  $H_n$ are the Hermite polynomials defined by ${H}_{-1}(\xi)=0$, ${H}_0(\xi)=1$ and for $n\geq
1$, ${H}_n(\xi)$ has the following recursive relation 
$$
\sqrt{n}\,{H}_n(\xi) \,=\, \xi \,{H}_{n-1}(\xi)-\sqrt{n-1}\,{H}_{n-2}(\xi)\,, \quad \forall \,n \geq 1\,.
$$
Let us also emphasize that
$H_n^\prime(\xi)\,=\,\sqrt{n}\,H_{n-1}(\xi)$ for all $n\geq 1$, and the set of functions $(\Psi_n)_n$ defined by \eqref{hbasisf} is an orthogonal system satisfying
\begin{equation}
  \label{proportho}
\left\langle\Psi_n,\Psi_m\right\rangle_{L^2(\dw)}\,=\,\alpha(t)\,\int_\RR \Psi_{n}(v)\,
  H_m(\alpha(t)\, v) \dD v \,=\,\alpha(t)\,\delta_{n,m},
\end{equation}
where $\delta_{n,m}$ is the Kronecker delta function. 
Finally, for any integer $N\geq 1$ and $t\geq 0$, we introduce the
space $V_N$  as the subspace
of $L^2(\dw)$ defined by
\begin{equation}
  \label{def_VN}
V_{N} \,:=\,\text{Span}\{\Psi_n(t),\quad 0 \leq
n\leq N-1 \}.
\end{equation}
Then we look for an approximate solution $f_N$ of \eqref{vlasov} as a finite sum which corresponds to a truncation of a
series
\beq
\label{fseries}
f_{N}(t,x,v)=\sum_{n=0}^{N-1}C_n(t,x)\,\Psi_n(t,v)\,,
\eeq
where $N$ is the number of modes and $(C_n)_{0\leq n\leq N-1}$ are  computed  using \textcolor{black}{the} orthogonality property \eqref{proportho}, and taking
$H_n(\alpha\,v)$ as test function in \eqref{vlasov}. Therefore,  a system of
evolution equations is obtained for the modes $(C_n)_{0\leq n< N }$ as
in \cite{BCF2021},
\begin{equation}
  \label{cn}
  \left\{
    \begin{array}{l}
\ds\partial_t
      C_n \,+\, \mT_n[C] \,=\, \mS_n[C,E_N] \,,
      \\[1.1em]
  \ds    \mT_n[C] \,=\,  \frac{1}{\alpha}\left(\sqrt{n}\,\partial_x
      C_{n-1}\,+\,\sqrt{n+1}\,\partial_xC_{n+1}\right)\,,
      \\[1.1em]
      \ds\mS_n[C,E_N] \,=\,  \frac{\alpha'}{\alpha}\,\left(n\,C_n+\sqrt{(n-1)n}\,C_{n-2}\right)\,+\,E_{N}\,\alpha\,\sqrt{n}\,C_{n-1}\,,
    \end{array}
  \right.
    \end{equation}
    with  $C_{n}=0$ when $n<0$ and $n\geq N$, and the initial data
    $C_n(t=0)$ is given by
    $$
C_n(t=0) \,=\, \frac{1}{\alpha(0)}\,\left\langle f_0,
  \,\Psi_n(0)\right\rangle_{L^2(\omega(0)\dD v)}\,.
    $$
    Meanwhile, we observe that the density $\rho_{N}$ satisfies
$$
\rho_{N} \,=\, \int_\RR f_{N} \,\dD v \,=\, C_0\,,
$$
and then the Poisson equation becomes 
\beq
\label{ps}
\f{\partial E_{N}}{\px} = C_0 - \rho_{0,N}\,,
\eeq
with $\rho_{0,N}$ such that
$$
\int_{\TT} \left(C_{0} - \rho_{0,N}\right) \,\dD x = 0\,. 
$$
Observe that when we take $N=\infty$ in the expression
\eqref{fseries}, we get an infinite system \eqref{cn}-\eqref{ps} of equations for $(C_n)_{n\in\NN}$ and $E_{N}$, which is formally equivalent to the Vlasov-Poisson  system \eqref{vlasov}.



\subsection{Spatial discretization}

As in \cite{BCF2021}, we consider a discontinuous Galerkin
approximation for the Vlasov equation with Hermite spectral basis in
velocity \eqref{cn}. Let us first introduce some notations and start
with $\mJ=\{0,\ldots, N_x-1\}$ describing  the set  of subintervals and
$\hat\mJ=\{0,\ldots,N_x\}$ related to the number of edges, where $N_x\geq 1$ is an integer, then we
consider  the set $\{\xR\}_{j\in\hat\mJ}$, a partition of the torus $\TT$, where each element is denoted as $I_j=[\xL,
\xR]$ with its length $h_j$  for $j\in\mJ$, and $h=\max_j h_j$. Finally, we introduce
the parameter $\delta=(h,1/N)$ related to the numerical discretization
in space and velocity.

Given any $k\in\mathbb{N}$, we define a finite dimensional discrete piecewise polynomial space
\begin{equation}
  \label{def:Xh}
X_h\,=\,\left\{u\in L^2(\TT):\, u|_{I_j}\in \PP_k(I_j), \quad j\in\mJ\right\}\,,
\end{equation}
where the local space $\PP_k(I)$ consists of polynomials of degree at
most $k$ on the interval $I$. We further denote the jump $[u]_\iR$
and the average $\{u\}_\iR$ of $u$ at $x_\iR$ defined as
$$
[u]_\iR\,=\,{u(x_\iR^+)\,-\,u(x_\iR^-)} \quad{\rm and}\quad
\{u\}_\iR\,=\,\frac12\,\left(u(x_\iR^+)\,+\,u(x_\iR^-)\right)\,, \quad \forall\, j\in\hat\mJ\,,
$$
where $u(x^\pm)=\lim_{\Delta x\rightarrow 0^\pm} u(x+\Delta x)$.  We
also denote

\[  u_\iR=u(x_\iR)\,, \qquad u^\pm_\iR=u(x^\pm_\iR)\,, \quad \forall\, j\in\hat\mJ\,. \]

From these notations, we apply a semi-discrete discontinuous Galerkin
method for \eqref{cn} as follows. We look for an approximation
$C_{\delta}=(C_{\delta,n})_{0\leq n \leq N-1}$ with $C_{\delta,n}(t,\cdot) \in X_h$, such that for any $\varphi_n \in X_h$, we have
\beq
  \frac{\dD}{\dD t}\int_{I_j}C_{\delta,n}\,\varphi_n\,\dD
  x\,+\,\mA_{n,j}(g_n(C_\delta),\varphi_n) \,=\, \int_{I_j}
     \mS_n[C_\delta,E_\delta]\,\varphi_n \,\dD x,\quad{\rm for }\quad j\in\mJ, \quad 0\leq n \leq N-1,\,
     \label{dgcn}
\eeq
where $\mA_{n,j}$ is defined by 
\beq
\label{anh}
\left\{
  \begin{array}{l}
\ds \mA_{n,j}(g_{n},\varphi_n) \,=\, -\int_{I_j}g_{n}\,\testR^\prime_n \,
    \dD
    x\,+\,\hat{g}_{n,j+\f12}\,\testR^-_{n,j+1/2}-\hat{g}_{n,j-1/2}\,\testR^+_{n,j-1/2}\,,
    \\[1.1em]
\ds g_{n}(C_\delta) \,=\,\frac{1}{\alpha}\left(\sqrt{n}\,C_{\delta,n-1}\,+\,\sqrt{n+1}\,C_{\delta,n+1}\right)\,.
\end{array}\right.
\eeq

The numerical flux $\hat{g}_{n}$ in \eqref{anh} is given by
\beq
\label{lfflx}
\hat{g}_{n}\,=\,\f12\left[g^-_{n}(C_\delta)+g^+_{n}(C_\delta)\,-\,\frac{\nu_n}{\alpha}\,\left(C^+_{\delta,n}\,-\,C^-_{\delta,n}\right)\right]\,,
\eeq
with the numerical viscosity coefficient $\nu_n$ such that
$\nu_n\in [\underline{\nu},\overline{\nu}]$ with $0<\underline{\nu}\leq \overline{\nu}<\infty$.


Therefore the approximate solution of \eqref{vlasov} obtained using Hermite polynomials in velocity variable and discontinuous Galerkin discretization in space is then defined by 
\beq
\label{dgfseries}
f_{\delta}(t,x,v)=\sum_{n=0}^{N-1}C_{\delta,n}(t,x)\,\Psi_n(t,v)\,,
\eeq
where $\delta=(1/N,h)$ is a small parameter,  $(C_{\delta,n})_n$ satisfy \eqref{dgcn} and $(\Psi_n)_n$ are the basis functions defined by \eqref{hbasisf}.

We now deal with the approximation $E_{\delta}$ of the electric field. To this end, we consider the potential function $\Phi_{\delta}(t,x)$, such that
\begin{equation}
  \label{approxPoisson}
\left\{
\begin{array}{l}
\ds E_{\delta}\,=\,-\f{\partial \Phi_{\delta}}{\px}\,, \\[0.9em]
\ds\f{\partial E_{\delta}}{\px} \,=\,C_{\delta,0} - \rho_{0,\delta}\,.
\end{array}\right.
\end{equation}
Hence we get the one dimensional Poisson equation 
$$
-\f{\partial^2 \Phi_{\delta}}{\px^2} \,=\,C_{\delta,0} \,-\, \rho_{0,\delta}\,,
$$
with $\rho_{0,\delta}$ such that
$$
\int_{\TT} \left(C_{\delta,0} - \rho_{0,\delta}\right) \,\dD x = 0\,. 
$$
We simply consider a conforming approximation of the electric potential, corresponding to a direct integration of this Poisson problem \eqref{approxPoisson}, which is straightforward in 1D.

\textcolor{black}{In what follows, we study the scheme \eqref{dgcn}--\eqref{approxPoisson}, where $\alpha$ is a time-dependent function defined in the next subsection by \eqref{eq:def_alpha}, and not a constant scaling parameter as usual. In \cite{BCF2021}, we provide a study of the conservation properties satisfied by this type of discontinuous Galerkin/Hermite spectral methods with a time-dependent scaling function $\alpha$. It appears that the conservation properties only rely on the choice of the spatial discretization and not on the definition of $\alpha$. Indeed, we proved in \cite[Proposition 2.1]{BCF2021} the conservation of mass, momentum and total energy for the Hermite velocity discretization \eqref{cn}. 
Then, concerning the spatial discretization, we established in \cite[Theorem 3.4]{BCF2021} the conservation of the discrete total energy for the scheme \eqref{dgcn}--\eqref{lfflx} with a centered numerical flux $\hat{g}_0$ (corresponding to $\nu_0=0$ in \eqref{lfflx}) together with a discontinuous Galerkin approximation of the Poisson equation. 
}

\subsection{Discussion on the scaling function $\alpha$ and main results}
Before to state an error estimate on the numerical solution to
\eqref{dgcn}-\eqref{approxPoisson}, let us introduce the suitable
functional framework. We set $\mu_t$ the measure given as 
\begin{equation}\label{def_dmu}
\dD \mu_t =
\omega(t,v)\dD x\dD v
\end{equation}
where the weight $\omega$ is provided in \eqref{eq:def_omega} and the following $L^2$ weighted space given by
\begin{equation*}
L^2(\dD\mu_t)\,:=\,\left\{g : \TT\times\RR\to\RR : \iint_{\TT\times\RR}
  |g(x,v)|^2\dD \mu_t <+\infty\right\},
\end{equation*}
with $\langle\cdot,\cdot\rangle_{L^2(\dD \mu_t)}$ the associated inner product, that is
\begin{equation*}
\langle f,g\rangle_{L^2(\dD\mu_t)} \,=\,\iint_{\TT\times\RR}f(x,v)\,g(x,v)\,\dD\mu_t\,,
\end{equation*}
and $\|\cdot\|_{L^2(\dD\mu_t)}$  the corresponding norm.

\textcolor{black}{Let us remark that for all functions $\alpha,\,\tilde{\alpha}:\RR_+\to\RR_+$ such that $\tilde{\alpha}(t)\leq\alpha(t)$ for all $t\geq 0$, the associated $\omega,\,\tilde{\omega}$ defined by \eqref{eq:def_omega} satisfy $\tilde{\omega}(t,v)\leq \omega(t,v)$ for all $t\geq 0,\,v\in\RR$. Then, the corresponding weighted $L^2$ norms verify
\[\int\int_{\TT\times\RR}|f(x,v)|^2\,\tilde{\omega}(t,v) \dD x\,\dD v\,\leq\, \int\int_{\TT\times\RR}|f(x,v)|^2\,\omega(t,v) \dD x\,\dD v.\]
In particular, taking $\tilde{\alpha}= 0$, the standard $L^2(\dD x\dD v)$ norm is controlled by the weighted $L^2$ norm:
\[\|f\|_{L^2(\dD x\,\dD v)}^2\,=\, \int\int_{\TT\times\RR}|f(x,v)|^2\,\dD x\,\dD v \,\leq\,  \int\int_{\TT\times\RR}|f(x,v)|^2\,\omega(t,v) \dD x\,\dD v \,=\,\|f\|_{L^2(\dD\mu_t)}^2.\]
 }

The first issue is to find the appropriate framework for the stability of approximations based
on asymmetri\-cally-weighted Hermite basis. Indeed, this choice fails to
preserve \textcolor{black}{the} $L^2$ norm of the approximate solution, and therefore to ensure \textcolor{black}{the}
long-time stability of the method.  Consequently, we introduce  a $L^2$ weighted space, with a time-dependent weight,
allowing to prove \textcolor{black}{the} global stability of the solution in this
space \cite{BCF2021}. Actually, this idea has been already employed in \cite{Ma2005,
  Ma2007} to stabilize Hermite spectral methods for linear diffusion
equations and nonlinear convection-diffusion equations in unbounded
domains, yielding stability and spectral convergence of the considered
methods. The main point now is to determine a function $\alpha$. \textcolor{black}{We proved the following result in \cite[Proposition 3.2]{BCF2021}.
\begin{prop}
\label{prop:stab_L2dvdx}
Let $(f_\delta,E_\delta)$ be the approximate solution defined by \eqref{dgcn}--\eqref{approxPoisson}, with the scaling function $\alpha$ defined by
\begin{equation}
  \label{eq:def_alpha}
\alpha(t)\,:=\,\alpha_0\,\left(1+\gamma\,\alpha_0^2\,\int_0^t\max(1,\|E_\delta(s)\|^2_{L^\infty}) \dD s\right)^{-1/2}.
\end{equation}
Assume that $\|f_\delta(0)\|_{L^2(\dD\mu_0)}<+\infty$. Then, for any $t\geq 0$, we have
\begin{equation}
\frac{\dD}{\dD t}\|f_{\delta}(t)\|_{L^2(\dD\mu_t)}^2\,:=\,\frac{\dD}{\dD t}
\left(\alpha(t)\,\sum_{n=0}^{N-1} \int_\TT |C_{\delta,n}|^2 \dD x 
\right)\,\leq\,-\sum_{n=0}^{N-1}\sum_{j\in\hat\mJ}\nu_n\,[C_{\delta,n}]_{j-\frac{1}{2}}^2+ \frac{1}{2\,\gamma}\,\|f_{\delta}(t)\|_{L^2(\dD\mu_t)}^2,
\end{equation}
from which we deduce
  \begin{equation}
 \label{eq2:stabL2dvdx}
\|f_{\delta}(t)\|_{L^2(\dD\mu_t)}\,\leq\, \|f_{\delta}(0)\|_{L^2(\dD\mu_0)}\,e^{t/4\gamma}\,,
\end{equation}
where $\gamma>0$ is the fixed parameter, which can be chosen arbitrarily, appearing in the definition \eqref{eq:def_alpha} of $\alpha$.
\end{prop}
To study the convergence of the numerical method, we also need to establish a stability result for the exact solution $f$ in the weighted norm $L^2(\dD \mu_t)$, where the weight depends on the approximate solution (see  definition \eqref{eq:def_alpha} of $\alpha$).
\begin{prop}
  \label{prop:1}
  Consider $(f,E)$ a smooth solution of the Vlasov-Poisson system
  \eqref{vlasov}. Assuming that the initial data $f_0$ belongs to $L^2(\dD\mu_0)$, then there
  exists $c_0>0$ such that the solution $f(t)$ satisfies for all $t \geq 0$:
 \begin{equation*}
\|f(t)\|_{L^2(\dD\mu_t)}\leq \|f_{0}\|_{L^2(\dD\mu_0)}\, e^{t/4\eta},
\end{equation*}
where $\dD\mu_t$ is defined by \eqref{def_dmu}, with $\alpha$ appearing in the weight $\omega$ given by \eqref{eq:def_alpha}.\\
The parameter $\eta > 0$ has to be chosen small enough, namely such that $\eta\,C/\gamma<1$, where $C>0$ is a constant depending only on the mass of $f_0$ such that $\|E(t)\|_{L^\infty}^2\leq C$ for all $t\geq 0$, and $\gamma>0$ is the fixed parameter appearing in the definition \eqref{eq:def_alpha} of $\alpha$.
\end{prop}
\begin{proof}
Using the Vlasov equation \eqref{vlasov} and the definition \eqref{eq:def_omega} of the weight $\omega$, we have
\[\frac{1}{2}\frac{\dD}{\dD t}\|f(t)\|_{L^2(\dD\mu_t)}^2\, = \, \frac{1}{2}\int\int_{\TT\times\RR}f^2\,(\alpha^2 \,E\,v+\alpha\,\alpha'\,|v|^2)\omega\,\dD x\,\dD v.
\]
Applying now the Young inequality on the first term, we get for $\eta>0$,
\[\frac{1}{2}\frac{\dD}{\dD t}\|f(t)\|_{L^2(\dD\mu_t)}^2\,\leq \, \frac{1}{2}\left(\frac{\eta}{2}\|E\|_{L^\infty}^2\alpha^4+\alpha'\alpha\right)\int\int_{\TT\times\RR}f^2\,|v|^2\omega\,\dD x\,\dD v+\frac{1}{4\,\eta}\,\|f(t)\|_{L^2(\dD\mu_t)}^2.
\]
Then, using the definition \eqref{eq:def_alpha} of $\alpha$, it is clear that
\[ \alpha'=-\frac{\gamma}{2}\,\max(1,\|E_\delta\|_{L^\infty}^2)\,\alpha^3,\]
leading to
\[\frac{1}{2}\frac{\dD}{\dD t}\|f(t)\|_{L^2(\dD\mu_t)}^2\,\leq \, \frac{1}{2}\left(\frac{\eta}{2}\|E\|_{L^\infty}^2-\frac{\gamma}{2}\max(1,\|E_\delta\|_{L^\infty}^2)\right)\alpha^4\int\int_{\TT\times\RR}f^2\,|v|^2\omega\,\dD x\,\dD v+\frac{1}{4\,\eta}\,\|f(t)\|_{L^2(\dD\mu_t)}^2.
\]
In one space dimension, the $L^1$ estimate on $f$ can be used to bound the electric field, namely there exists a constant $C>0$ depending only on the initial mass of $f$ such that for all $t\geq 0$, $\|E(t)\|_{L^\infty}^2\leq C$.\\
Then, choosing $\eta>0$ such that $\eta\,C/\gamma<1$, the first term of the right-hand side is nonpositive, which concludes the proof.
\end{proof}
Notice that the functional space $L^2(\dD \mu_t)$ depends on $\alpha$ given by \eqref{eq:def_alpha}, and then on the discretization parameter itself through the term $\|E_\delta(t)\|_{L^\infty}$ involved in this definition. Therefore, to establish a convergence result, it is mandatory to control $\alpha$ uniformly with respect to $\delta$. This control is achieved by bounding $\|E_\delta(t)\|_{L^\infty}$ uniformly with respect to $\delta$.
\begin{prop}\label{prop:bornes_alpha}
We consider a solution $(f_\delta,E_\delta)$ of \eqref{dgcn}--\eqref{approxPoisson}, with $\alpha$ defined by \eqref{eq:def_alpha}. We assume that $\|f_\delta(0)\|_{L^2(\dD\mu_0)}<+\infty$. Let $T>0$ be a fixed final time.\\
Then, there exists a constant $C_T>0$ independent of the discretization parameter $\delta$ such that 
\begin{equation*}
\|E_\delta(t)\|_{L^\infty}\leq C_T, \quad \forall t\in [0,T].
\end{equation*}
Moreover, the scaling function $\alpha$ satisfies
\begin{equation}\label{control_alpha}
0<\underline{\alpha}_T\leq \alpha(t)\leq \alpha_0, \quad\forall t\in [0,T],
\end{equation}
where the constant $\underline{\alpha}_T$ is independent of $\delta$ and given by
\begin{equation*}
\underline{\alpha}_T\,:=\,\alpha_0\,(1+\gamma\,\alpha_0^2(1+C_T)\,T)^{-1/2}.
\end{equation*}
\end{prop}
\begin{proof}
The proof of the uniform $L^\infty$ bound on $E_\delta$ is mainly based on the Sobolev and Poincaré-Wirtinger inequalities, together with the stability estimate \eqref{eq2:stabL2dvdx}, as detailed in \cite[Theorem 3.6]{BCF2021}. Thanks to this bound, it is straightworfard to obtain the lower bound of $\alpha$ by using the definition \eqref{eq:def_alpha}. The upper bound $\alpha_0$ is also clear since $\alpha$ is a nonincreasing function.
\end{proof}
We are now in position to state our main result.
\begin{thm}
  \label{th:main}
For any $t\in [0,T]$,  consider the scaling function $\alpha$
defined by \eqref{eq:def_alpha} and let $f(t,.) \in \HH^m(\dD\mu_t)$ be the solution
of the Vlasov-Poisson system \eqref{vlasov} where $m \geq k+1$ and
$f_{\delta}$ be the approximation defined by
\eqref{dgcn}-\eqref{dgfseries}. Then there exists a constant $\mC>0$,
independent of $\delta\,=\,(h,1/N)$, but depending on $T$, such that
\begin{equation}
  \label{res:main}
\|f(t)-f_{\delta}(t)\|_{L^2(\dD x\,\dD v)}\leq\|f(t)-f_{\delta}(t)\|_{L^2(\dD\mu_t)}\leq
\mC\,\left[\frac{1}{N^{(m-1)/2}}\,+\, h^{k+1/2}\right].
\end{equation}
\end{thm}
}

Before to present the proof of this result, let us make some
comments.

\begin{itemize}
\item On the one hand this result shows that the discretization in velocity
using Hermite polynomials provides  spectral accuracy. On the other
hand, we get the classical order of convergence of
the discontinuous Galerkin method for the space
discretization.

\item For \textcolor{black}{the} sake of clarity, we do not write explicitly how the
error bound \eqref{res:main} depends on the scaling parameter $\alpha$, nevertheless in the proof we will follow carefully this dependence.

%
\end{itemize}

\section{Preliminary results}
\setcounter{equation}{0}
\label{sec:3}

In the next section, we provide some results about the propagation of
regularity of the solution to the Vlasov-Poisson system \eqref{vlasov}.

\subsection{Propagation of the weighted Sobolev norms}

The global existence of classical $\CC^1$ solutions of the Vlasov-Poisson
system has been obtained by Ukai-Okabe \cite{Ukai} in the case of space
dimension two, and by Pfaffelmoser \cite{Pfaf} and Lions-Perthame \cite{LP}
independently in three dimensions. The key to obtain classical
solutions of the Vlasov-Poisson system is to prove that the
macroscopic (charge) density $\rho \in L^\infty([0, T ] \times \RR^d )$
for all $T > 0$.  Following Ukai \& Okabe who show the decay
of the distribution function $f$ in the $v$ variable  formulated in
terms of a convenient weighted estimate, we get \textcolor{black}{the} propagation of $\CC^{m}$
regularity for the  solution $(f,E$). \textcolor{black}{More precisely, according for example to the article of Ukai \& Okabe \cite{Ukai}, the following result holds.} 

Let $\beta \in  \CC^m(\RR)$ be such that
$$
\beta\geq 0, \quad \beta^\prime \leq 0, \quad{\rm and} \quad
\beta(r)=O\left(\frac{1}{r^\lambda}\right)\,,
$$
with $\lambda >1$.

\begin{thm}
For any $m\geq 1$, assume that $f_0$ is nonnegative such that $f_0\in \CC^{m-1}(\TT\times\RR)$ and for
all $0\leq k \leq m$,
$$
 \left( \,|\partial_x^k f_0| \,+\,
|\partial_v^k f_0|\,\right)(x,v)\,\leq\, \beta(|v|)\,.
$$
Then there exists a unique classical solution to the Vlasov-Poisson
system \eqref{vlasov} satisfying for any $0\leq k\leq m-1$
$$
\left( \,|\partial_x^k f| \,+\,
|\partial_v^k f|\,\right)(t,x,v) \,=\,
O\left({1}/{|v|^\lambda}\right), \qquad {\rm as}\, |v|\rightarrow \infty,   
$$
uniformly in $(t,x)\in[0,T]\times\TT$. Moreover  the electric
field $E$ satisfies for all $1\leq k \leq m$
$$
\left( \, |E| \,+\, |\partial_x^k E|\,\right)(t,x) \,\leq \, \mC_T.
$$
\label{th:1}
\end{thm}
\textcolor{black}{The proof of this result is done in the two dimensional case in \cite{Ukai}, but also applies in the simpler one dimensional~case.}

This latter result allows to obtain uniform estimates on the electric
field and its space derivative in $\CC^{m}([0,T]\times\TT)$, hence
we can propagate the weighted Sobolev norm on the solution to the
Vlasov-Poisson system \eqref{vlasov}
\begin{cor}
  Under the assumption of Theorem \ref{th:1} and assuming that
  $$
\| f_0\|_{\HH^m(\dD\mu_0)}  < \infty,
$$
we have that for any $t\in [0,T]$
$$
\|f(t) \|_{\HH^m(\dD\mu_t)}   \leq \| f_0\|_{\HH^m(\dD\mu_0)}   \,\exp(\mC\, t).
$$
\end{cor}
\begin{proof}
  Theorem \ref{th:1} ensures that the electric field is such that for
  all $t\in([0,T] $
  $$
\|E(t)\|_{W^{m,\infty}} \,\leq\, \mC_T\,.
$$
Then since the electric field is uniformly bounded, we  propagate \textcolor{black}{the}
$H^m(\dD\mu_t)$ norms of $f(t)$ as we estimate the weighted $L^2(\dD\mu_t)$ norm  in Proposition \ref{prop:1}.
\end{proof}

\subsection{Projection error of the Hermite decomposition}
In this section we present some approximation properties of the chosen
Hermite functions. Since the results presented here are very similar
to those proposed in \cite[Section 2]{Fok2001}, we only briefly
outline the proofs.

For any  integer $m\geq 1$, we define
\[\HH^m(\dw):=\left\{g:\RR\to\RR\, ;\, \partial^l_v g\in L^2(\dw),\, 0\leq l\leq m\right\},\]
with the following seminorm and norm:
\[|g|_{\HH^m(\dw)}=\left\|\partial^m_v g\right\|_{L^2(\dw)},\qquad \|g\|_{\HH^m(\dw)}=\left(\sum_{l=0}^m|g|_{\HH^l(\dw)}^2\right)^{{1}/{2}}.\]

Using \textcolor{black}{the} definition \eqref{hbasisf} of $\Psi_n$ and \textcolor{black}{the} properties of the Hermite polynomials $H_n$, one obtains that for all $n\geq 0$,
\beq
\label{vPsi}
\left\{\begin{array}{l}
         \ds\partial_t\Psi_n\,=\,-\frac{\alpha'}{\alpha}\left(n\,\Psi_n\,+\,\sqrt{(n+1)(n+2)}\,\Psi_{n+2}\right)\,,
         \\[1.1em]
\ds
         \alpha \,v\,\Psi_n\,=\,\sqrt{n+1}\,\Psi_{n+1}\,+\,\sqrt{n}\,\Psi_{n-1}\,,
         \\[1.1em]
\ds\partial_v\Psi_n\,=\, -\alpha\,\sqrt{n+1}\,\Psi_{n+1}\,. 
       \end{array}\right.
\eeq
Using the latter relation,  the set $(\partial_v\Psi_n)_n$ is also an orthogonal system, namely
\begin{equation}
  \label{proporthodv}
\left\langle\partial_v\Psi_n,\,\partial_v\Psi_m\right\rangle_{L^2(\dw)}\,=\,\alpha^3(t)\,(n+1)\,\delta_{n,m}\,.
\end{equation}
Therefore, any $g \in L^2(\dw)$ can be expanded as
\begin{equation}
  \label{decompo_u_v}
 g(v)\,=\,\sum_{n\in\NN}\hat{g}_n(t)\,\Psi_n(t,v),
 \end{equation}
 and the Hermite coefficients are given by
\begin{equation}\label{coeff_u}
\hat{g}_n(t)\,=\,\frac{1}{\alpha(t)} \,\left\langle g,\,\Psi_n(t)\right\rangle_{L^2(\dw)}.
\end{equation}
Remark also that the following equalities hold for every $g\in L^2(\dw)$:
\begin{equation}\label{expressions_norme}
\|g\|_{L^2(\dw)}^2\,=\,\alpha(t)\,\sum_{n\in\NN}
|\hat{g}_n(t)|^2\,=\frac{1}{\alpha(t)}\sum_{n\in\NN} \left|\left\langle g,\,\Psi_n(t)\right\rangle_{L^2(\dw)}\right|^2\,.
\end{equation}
Finally, we also introduce $\mP_{V_N}$ the orthogonal projection on $V_N$  such that we have  
\begin{equation*}
 \left\langle g\,-\,\mP_{V_N}g ,\,\varphi\right\rangle_{L^2(\dw)}\,=\,0,\qquad \forall \varphi\in V_{N},
 \end{equation*} 
 and
 \[
   \mP_{V_N}g(t,v) \,=\, \sum_{n=0}^{N-1}\hat{g}_{n}(t)\,\Psi_n(t,v)\,.
 \]
 Following \cite{Fok2001}, we now establish some inverse inequalities and imbedding inequalities which are needed to analyze the spectral convergence property for the here considered Hermite method.

\begin{lem}
For any $g \in V_{N}$, 
\begin{equation*}
|g|_{\HH^1(\dw)}\,\leq\, \alpha(t)\,\sqrt{N}\,\|g \|_{L^2(\dw)}.
\end{equation*}
\end{lem}
\begin{proof}
Decomposing $g\in V_{N}$ as
\[
  g(t,v)\,=\,\sum_{n=0}^{N-1}\hat{g}_n(t)\,\Psi_n(t,v)\,,
\]
and using \textcolor{black}{the} orthogonality properties \eqref{proportho} and \eqref{proporthodv} yields
\begin{equation*}
| g |_{\HH^1(\dw)}^2\,=\,\alpha^3(t)\,\sum_{n=0}^{N-1}(n+1)\,|\hat{g}_n(t)|^2\,\leq \, \alpha^3(t)\,N\,\sum_{n=0}^{N-1}|\hat{g}_n(t)|^2\,=\,\alpha^2(t)\, N\,\|g\|_{L^2(\dw)}^2.
\end{equation*}
\end{proof}

Furthermore, we also show the following inequalities.
\begin{lem}
  \label{lemme2}
For any $g\in \HH^1(\dw)$, it holds
\beq
\label{lemme2_ineg1}
\left\{\begin{array}{l}
\ds  \|g\|_{L^2(\dw)}\,\leq\, \frac{2}{\alpha(t)}\,|g|_{\HH^1(\dw)}\,,
                    \\[1.1em]
\ds \|v\,g\|_{L^2(\dw)}\,\leq\, \frac{2}{\alpha^2(t)}\,|g|_{\HH^1(\dw)}\,.
         \end{array}\right.
       \eeq
\end{lem}

\textcolor{black}{
\begin{proof}
Using that $\partial_v\omega(t,v)\,=\,v\,\alpha^2(t)\,\omega(t,v)$, we have
\[\int_{\RR}(v\,g(v))^2\,\omega(t,v)\,\dD v\,=\,\frac{1}{\alpha^2(t)}\int_{\RR}v\,g(v)^2\,\partial_v\omega(t,v)\,\dD v.\]
By an integration by parts, one gets
\begin{equation*}
\int_{\RR}(v\,g(v))^2\,\omega(t,v)\,\dD v\,=\, -\frac{1}{\alpha^2(t)}\int_{\RR}g(v)^2\,\omega(t,v)\,\dD v-\frac{2}{\alpha^2(t)}\int_{\RR}v\,g(v)\,\partial_vg(v)\,\omega(t,v)\,\dD v.
\end{equation*}
Then, applying the Cauchy-Schwarz inequality to the second term of the right hand side, it yields
\begin{equation}\label{estim_vu2}
\|v\,g\|_{L^2(\omega(t)\,\dD v)}^2+\frac{1}{\alpha^2(t)}\,\|g\|_{L^2(\omega(t)\,\dD v)}^2\,\leq\, \frac{2}{\alpha^2(t)}\,\|v\,g\|_{L^2(\omega(t)\,\dD v)}\,|g|_{H^1(\omega(t)\,\dD v)},
\end{equation}
from which we deduce that
\[ \|v\,g\|_{L^2(\omega(t)\,\dD v)}\,\leq\, \frac{2}{\alpha^2(t)}\,|g|_{H^1(\omega(t)\,\dD v)}.\]
Now, using this later estimate in \eqref{estim_vu2}, we obtain that
\[\|g\|_{L^2(\omega(t)\,\dD v)}^2\,\leq\, \frac{4}{\alpha^2(t)}\,|g|_{H^1(\omega(t)\,\dD v)}^2,\]
which concludes the proof.
\end{proof}
}

Still following \cite{Fok2001}, let us define the Fokker-Planck operator $\mF$ as
\begin{equation}\label{operatorA}
\mF[g](v)\,=\,-\partial_v\left(\omega^{-1}(t)\,\partial_v\left(g\,\omega(t)\right)\right)\,,
\end{equation}
where $\omega(t)$ is the weight defined in \eqref{eq:def_omega}. It follows from Lemma \ref{lemme2} that $\mF$ is a continuous mapping from $\HH^2(\dw)$ to $L^2(\dw)$, with $\|\mF\|$ independent of $t$, as stated in the following lemma.

\begin{lem}
  \label{lem_Acont}
For all $g\in \HH^2(\dw)$, it holds
\begin{equation}\label{norm_opA}
\|\mF[g]\|_{L^2(\dw)}\,\leq\, 7\,|g|_{\HH^2(\dw)}.
\end{equation}
\end{lem}
\begin{proof}
Using the definition \eqref{operatorA}, we obtain after some computations that for all $g\in \HH^2(\dw)$,
\[
  \mF[g] (v)\,=\,-\partial^2_vg \,-\,\alpha^2(t)\,v\,\partial_v
  g\,-\,\alpha^2(t)\,g.
\]
Then, applying \textcolor{black}{the} triangle inequality together with the second inequality
of Lemma \ref{lemme2} on the second term and twice the first
inequality of Lemma \ref{lemme2} on the third term, we obtain the expected estimate.
\end{proof}

Moreover, $\Psi_n$ is the $n$-th eigenfunction of the following singular Liouville problem:
\begin{equation}
  \label{Liouville}
-\mF[g](v)\,+\,\lambda\,g(v)\,=\,0, \qquad v\in\RR,
\end{equation}
with corresponding eigenvalues $\lambda_n\,=\, \alpha^2(t)\,n$.

\begin{prop}
  \label{prop_interpolation_v}
Let $r\geq 0$. For any $g\in \HH^r(\dw)$, it holds for all $N\geq 0$
\begin{equation}
  \label{estim_interpolation}
\|\,g-\mP_{V_N}g\,\|_{L^2(\dw)}\,\leq\, \frac{\mC}{\left(\alpha^2(t)\,N\,\right)^{r/2}}\,\|g\|_{\HH^r(\dw)}\,,
\end{equation}
with $\mC>0$ independent of $N$ and $t$.
\end{prop}

\begin{proof}
Throughout this proof, let $\mC$ be a generic positive constant
independent of $N$ and $t$, which may be different in different
places.  Using the orthogonal relation \eqref{proportho} and the definition of the orthogonal projection $\mP_{V_N}$, we have
\begin{equation*}
\|\,g-\mP_{V_N} g\,\|_{L^2(\dw)}^2\,=\,\alpha(t)\,\sum_{n\geq N}|\hat{g}_n(t)|^2\,.
\end{equation*}
Let us first treat the case where $r$ is an even integer. By the singular Liouville equation \eqref{Liouville}, we get
\begin{equation*}
\int_{\RR}g\,\Psi_n\,\dw \,=\,\frac{1}{\alpha^2(t)\,n}\int_{\RR}g\,\mF[\Psi_n]\,\dw\,.
\end{equation*}
Then, using \textcolor{black}{the} definition of the Fokker-Planck operator $\mF$ in \eqref{operatorA} and performing two successive integrations by part\textcolor{black}{s}, it holds
\begin{eqnarray*}
\int_{\RR}g\,\Psi_n\,\dw
  &=&\frac{1}{\alpha^2(t)\,n}\int_{\RR} \omega^{-1}(t)\,\partial_v\left(\Psi_n\,\omega(t)\right)\,\partial_vg\,\dw\,,
  \\[1.1em]
&=& -\frac{1}{\alpha^2(t)\,n}\,\int_{\RR}\Psi_n
    \,\partial_v\left( \omega^{-1}(t)\,\partial_v\left(g\,\omega(t)\right)\right)\,\dw\,,
  \\[1.1em]
&=& \frac{1}{\alpha^2(t)\,n}\,\int_{\RR}\Psi_n\,\mF[g]\,\dw\,.
\end{eqnarray*}
Then by induction, we deduce that
\begin{equation}
  \label{estim1_thm1}
\int_{\RR} g\,\Psi_n\,\dw \,=\, \frac{1}{\left(\alpha^2(t)\,n\right)^{r/2}}\,\int_{\RR}\mF^{r/2}[g]\,\Psi_n\,\dw\,.
\end{equation}
Consequently, using \eqref{coeff_u}, it yields
\begin{equation}
  \label{estim2_thm1}
|\hat{g}_n(t)|=\frac{1}{\alpha(t)}\,\frac{1}{\left(\alpha^2(t)\,n\right)^{{r}/{2}}}\,\left|\int_{\RR}\mF^{{r}/{2}}[g]\,\Psi_n\,\dw\,\right|\,.
\end{equation}
Furthermore, using \eqref{expressions_norme} and Lemma \ref{lem_Acont}, we get
\begin{eqnarray*}
\|\;g-\mP_{V_N} g\|_{L^2(\dw)}^2&=& \alpha(t)\sum_{n\geq N}
                                  \frac{1}{\alpha^2(t)}\frac{1}{\left(\alpha^2(t)\,n\right)^{r}}\,\left|\int_{\RR}\mF^{{r}/{2}}[g]\,\Psi_n\,\dw\,\right|^2\,,
  \\[1.1em]
&\leq&
       \frac{1}{\left(\alpha^2(t)\,N\right)^{r}}\,\|\mF^{{r}/{2}}[g]\|_{L^2(\dw)}^2\,,
  \\[1.1em]
&\leq& \frac{\mC}{\left(\alpha^2(t)\,N\right)^{r}}\,\|g\|_{\HH^r(\dw)}^2\,.
\end{eqnarray*}
Now, let $r$ be any odd integer. Applying  \eqref{estim1_thm1} for $r-1$ (which is now even), using \textcolor{black}{the} Liouville equation \eqref{Liouville} and an integration by parts, we have
\begin{eqnarray*}
\int_{\RR}g\,\Psi_n\,\dw&=&
                            \frac{1}{\left(\alpha^2(t)\,n\right)^{{(r-1)}/{2}}}\,\int_{\RR}\mF^{{(r-1)}/{2}}[g]\,\Psi_n\,\dw
  \\[1.1em]
&=&
    \frac{1}{\left(\alpha^2(t)\,n\right)^{{(r+1)}/{2}}}\,\int_{\RR}\partial_v\left(\mF^{{(r-1)}{2}}[g]\,\omega(t)\right)\,\partial_v\left(\Psi_n\,\omega(t)\right)\omega^{-1}(t)\,\dD v.
    \label{estim3_thm1}
\end{eqnarray*}
On the one hand, by virtue of the two last  relations in \eqref{vPsi}, we obtain
\[
  \omega^{-1}(t)\,\partial_v\left(\Psi_n\,\omega(t)\right)\,=\,\alpha(t)\,\sqrt{n}\,\Psi_{n-1}\,.
\]
On the other hand, we remark that
\[
  \omega^{-1}(t)\,\partial_v\left(\mF^{(r-1)/2}[g]\,\omega(t)\right)\,=\,\partial_v\mF^{(r-1)/2}[g]\,+\,\alpha^2(t)\,v\,\mF^{(r-1)/2}[g]\,.
\]
Using these latter results, it yields that
\begin{equation*}
|\hat{g}_n|=\frac{1}{\alpha(t)}\,\frac{1}{\left(\alpha^2(t)\,n\right)^{r/2}}\,\left|\int_{\RR}\left(\partial_v\mF^{(r-1)/2}[g]\,+\,\alpha^2(t)\,v\,\mF^{(r-1)/2}[g]\right)\,\Psi_{n-1}\,\,\dw\,\right|.
\end{equation*}
Then, proceeding as above and using \eqref{lemme2_ineg1}, we get
\begin{eqnarray*}
\|g-\mP_{V_N} g\|_{L^2(\dw)}^2&\leq&
                                     \frac{\mC}{\left(\alpha^2(t)\,N\right)^{r}}\,\left[ \left\|\partial_v\mF^{(r-1)/2}[g]\right\|_{L^2(\dw)}^2\,+\,\alpha^4(t)\left\|v\,\mF^{{(r-1)}/{2}}[g]\right\|_{L^2(\dw)}^2\right]\,
  \\
&\leq& \frac{\mC}{\left(\alpha^2(t)\,N\right)^{r}}\, \left\|\partial_v
       \mF^{{(r-1)}/{2}}[g]\right\|_{L^2(\dw)}^2\,
  \\
&\leq & \frac{\mC}{\left(\alpha^2(t)\,N\right)^{r}}\, \|g\|_{\HH^r(\dw)}^2\,,
\end{eqnarray*}
which concludes the proof.
\end{proof}

\subsection{Projection error of the discontinuous Galerkin methods}

Now, let us recall a classical result concerning the spatial
approximation (see for example \textcolor{black}{\cite[Lemma 2.1]{johnson1986})}.

\begin{prop}
  \label{prop_interpolation_x}
There exists a constant $\bar{\mC}>0$, independent of $h$, such that for any $g\in \HH^{k+1}(\dD x)$, the following inequality holds:
\begin{equation}\label{estim_interpolation_x}
\|g-\mP_{X_h} g \|_{L^2}^2\,+\,h\,\|g-\mP_{X_h}g
\|_{L^2(\Gamma)}^2 \,\leq\, \bar{\mC}\,h^{2(k+1)}\,\|g\|^2_{\HH^{k+1}},
\end{equation}
where $\|\cdot\|_{L^2(\Gamma)}$ is the norm over the mesh skeleton $\Gamma\,=\,\left(x_{j+{1}/{2}}\right)_{j\in\hat\mJ}$ defined by
\beq
  \label{def_normGamma}
  \|u\|_{L^2(\Gamma)}^2\,=\,\sum_{j\in\hat\mJ}\left(|u_{j-{1}/{2}}^+|^2+|u_{j-{1}/{2}}^-|^2\right)\,.
\eeq
\end{prop}

\subsection{Global projection error}
In this section, we provide a global projection error on the
combination of Hermite polynomial approximations
and local discontinuous Galerkin interpolation in the suitable
functional space. For any $t\geq 0$, we introduce $W_{\delta}$, the subspace of $L^2(\dD\mu_t)$ defined by
$$
W_{\delta} \,:=\, V_N\otimes X_h,
$$
where $V_N$ is given by \eqref{def_VN} and $X_h$ is defined in
\eqref{def:Xh}, and  $\mP_{W_\delta}$ the orthogonal projection on
$W_{\delta}$. For $f\in L^2(\dD\mu_t)$ which can be expanded as
\begin{equation} \label{decompo_f}
f(x,v)=\sum_{n\in\NN}\hat{C}_n(t,x)\,\Psi_n(t,v),
\end{equation}
the projection $\mP_{W_\delta}f$ is then given by
\begin{equation}
  \label{decompo_fN}
\mP_{W_\delta} f\,=\,\sum_{n=0}^{N-1}\mP_{X_h}\hat{C}_n(t,x)\,\Psi_n(t,v),
\end{equation}
where $\mP_{X_h}$ is the $L^2(\dD x)$-orthogonal projection on
$X_h$. We prove the following result.

\begin{thm}
  \label{thm_errproj}
 For any $t\geq 0$, we consider  $g(t) \in
 H^r(\dD\mu_t)$, with $r\geq k+1$. Then we have
 \begin{itemize}
\item[$(i)$] the global projection error for all $t\in [0,T]$,
   \begin{equation*}
\|g(t)-\mP_{W_\delta} g(t)\|_{L^2(\dD\mu_t)}\,\leq\, \mC \,\left(
  \frac{1}{(\alpha^2(t) N)^{r/2}}\,+\,h^{k+1}\right)\,\|g(t)\|_{\HH^{r}(\dD\mu_t)}\,,
\end{equation*}
\item[$(ii)$] the global projection error on the fluxes for all $t\in [0,T]$,
  \begin{equation*}
\left(\int_\RR \|(g-\mP_{W_\delta}g)(t,\cdot,v)\|_{L^2(\Gamma)}^2\,\dw\right)^{1/2} \leq \mC \left(
  \frac{1}{(\alpha^2(t) N)^{r/2}}+ h^{k+1/2}\right)\,\|g(t)\|_{\HH^{r}(\dD\mu_t)}\,.
\end{equation*}
\end{itemize}
\end{thm}
\begin{proof}
To prove $(i)$,  we  first write $g(t)\in L^2(\dD\mu_t)$ as
 \begin{equation}
 g(t,x,v)\,=\,\sum_{n\in\NN}\hat{C}_n(t,x)\,\Psi_n(t,v),
 \end{equation}
where the modes $(\hat{C}_n)_{n\in\NN}$ are computed from $g(t)$
using the orthogonality property \eqref{proportho}. Hence, we  get 
\beq
\label{def:PW}
  g(t)-\mP_{W_\delta} g(t) \,=\,\sum_{n\geq N} \hat{C}_{n}\Psi_n
  \,+\,
  \sum_{n=0}^{N-1}\left(\hat{C}_{n}-\mP_{X_h}\hat{C}_n\right)\,\Psi_n.
  \eeq
On the one hand, observing that the first term of the right hand side is nothing else
than $g(t)-\mP_{V_N}g(t)$, it  can be estimated thanks to Proposition
\ref{prop_interpolation_v} 
$$
\|\,\sum_{n\geq N} \hat{C}_{n}\Psi_n \,\|_{L^2(\dD\mu_t)} \,\leq \, \frac{\mC}{\left(\alpha^2(t)\,N\,\right)^{r/2}}\,\|g\|_{\HH^r(\dD\mu_t)}\,.
$$
On the other hand, the second term can be controlled by applying
Proposition \ref{prop_interpolation_x} with $g=\hat{C}_n$, for $0\leq
n \leq N-1$, which yields
\begin{eqnarray*}
\|\,\sum_{n=0}^{N-1}\left(\hat{C}_{n}-\mP_{X_h}\hat{C}_n\right)\,\Psi_n\,\|_{L^2(\dD\mu_t)}^2
   &\leq& \bar{\mC}\,h^{2(k+1)}\,\alpha(t)\,
          \sum_{n=0}^{N-1} \|\,\hat{C}_{n}\,\|^2_{\HH^{k+1}(\dD x)}\,,
  \\
  &\leq& \bar{\mC}\,h^{2(k+1)}\,
          \|\, g\,\|^2_{\HH^{k+1}(\dD\mu_t)}\,.
\end{eqnarray*}
Gathering the latter results and using that $r\geq k+1$, we prove that there exists a constant
$\mC>0$, independent of the discretization parameter $\delta\,=\,(h,1/N)$,
such that  for all $t\in [0,T]$
  \begin{equation*}
\|g(t)-\mP_{W_\delta} g(t)\|_{L^2(\dD\mu_t)}\,\leq\, \mC \,\left(
  \frac{1}{(\alpha^2(t) N)^{r/2}}\,+\,h^{k+1}\right)\,\|g(t)\|_{\HH^{r}(\dD\mu_t)}\,.
\end{equation*}
The second estimate in $(ii)$ is obtained using the same
  ideas, together with the definition \eqref{def_normGamma} of the
  norm over the mesh skeleton and the classical trace theorem on
  Sobolev spaces. Indeed, from \eqref{def:PW} and Proposition
  \ref{prop_interpolation_v}, we obtain  that for all $x\in\TT$,
\begin{eqnarray*}
 \int_\RR |g(t,x,v)-\mP_{W_\delta} g(t,x,v)|^2\dw
 &\leq &
  \frac{\mC}{\left(\alpha^2(t)\,N\,\right)^{r}}\,\|g(t,x,.)\|_{\HH^{r}(\dw)}^2\\
  &+&  \alpha(t)\,\sum_{n=0}^{N-1} |\hat{C}_{n}-\mP_{X_h}\hat{C}_n|^2(t,x) \,.
\end{eqnarray*}
Furthermore, since $H^r(\TT)\subset L^\infty(\TT)$ for $r\geq 1$ and thanks to
Proposition \ref{prop_interpolation_x} on the mesh skeleton $\Gamma$, it yields
\begin{equation*}
  \int_\RR \|(g-\mP_{W_\delta} g)(t,.,v)\|^2_{L^2(\Gamma)}\,\dw
    \leq
\mC\,\left( \frac{1}{\left(\alpha^2(t)\,N\,\right)^{r}}\,+\,h^{2k+1}\right)\,\|g(t)\|_{\HH^r(\dD\mu_t)}^2,
\end{equation*}
from which we deduce the second item.
\end{proof}

%
%

\section{Proof of Theorem \ref{th:main}}
\setcounter{equation}{0}
\label{sec:4}

From the previous stability
analysis and global projection error,  we are now ready to prove our main result on the convergence of the
numerical solution $f_\delta$  given by \eqref{dgcn}-\eqref{dgfseries}
to the solution $f$ to the Vlasov-Poisson system \eqref{vlasov}.  By \textcolor{black}{the} triangle inequality, we have
\begin{equation}
  \label{ineg_trian}
 \|f(t)-f_{\delta}(t)\|_{L^2(\dD\mu_t)}\leq
 \|f(t)-\mP_{W_\delta}f(t)\|_{L^2(\dD\mu_t)}\,+\,\|\mP_{W_\delta} f(t)-f_{\delta}(t)\|_{L^2(\dD\mu_t)}\,,
 \end{equation}
 where the first term on the right hand side is the projection error,
 which  has already been estimated in Theorem \ref{thm_errproj},
 whereas  the second one is the consistency error,  which will be
 treated by considering its time derivative and using stability arguments together with interpolation properties.


We define the consi\textcolor{black}{s}tency error
$\mB$ for a given $C=(C_n)_{n\in\NN}$ and \textcolor{black}{a} smooth test function
$\varphi$ as
\begin{equation}
\mB_{n}(C,E,\varphi):=\int_\TT\partial_t C_n\,\varphi\,\dD
x\,+\,\sum_{j\in\mJ}\mA_{n,j}(g_n(C),\varphi) - \int_\TT\mS_n[C,E]\,\varphi \,\dD x\,,
\label{defBnh}
\end{equation}
where $(\mA_{n,j},g_n(C))$ are given in  \eqref{anh}  and $\mS_n[C,E]$ in \eqref{cn}.

On the one hand, from the modes $C_\delta=(C_{\delta,n})_{0\leq n< N}$
corresponding to  $f_{\delta}$ satisfy\textcolor{black}{ing} \eqref{dgcn}, we construct
$\bar{C}_\delta=(\bar{C}_{\delta,n})_{n\in\NN}$ as
$$
\bar{C}_{\delta,n} = \left\{
  \begin{array}{ll}
    C_{\delta,n},  & {\rm if }\quad  0 \leq n \leq N-1\,,
    \\
    0,  & {\rm else,} 
    \end{array}\right.
$$
which satisfies for any $n\in\NN$
\begin{equation}
  \label{FVapproxcompact}
\mB_{n}(\bar{C}_\delta,E_{\delta},\varphi_n)\,=\,0,\qquad\forall \varphi_n\in X_{h},
\end{equation}
where $E_{\delta}$ is  solution to \eqref{approxPoisson}.

On the other hand, by consistency of the numerical flux \eqref{lfflx},
the first modes $\hat{C}=(\hat{C}_n)_{n\in\NN}$ corresponding to
the exact continuous solution $f$ of \eqref{vlasov} satisfies for all $n\in\NN$
\begin{equation}
  \label{FVexactcompact}
\mB_{n}(\hat{C},E,\varphi_n)\,=\,0,\qquad\forall \varphi_n\in
X_{h}, 
\end{equation}
with $E$ given by the second equation (Poisson equation) of
\eqref{vlasov}. Then we introduce, for any $t\geq 0$,  $\eta_\delta(t) \in W_\delta$
as
\begin{equation}
  \label{def_eta}
\eta_\delta(t,x,v) \,:=\, \mP_{W_\delta}f(t,x,v)\,-\,f_{\delta}(t,x,v)\,=\,\sum_{n\in\NN}\eta_{\delta,n}(t,x)\,\Psi_n(t,v),
\end{equation}
with
$$
\eta_{\delta,n}\,=\,
\mP_{X_h}\hat{C}_n\,\mathbf{1}_{\{n<N\}}-\bar{C}_{\delta,n}, \quad n\in\NN,
$$
and $\xi_\delta(t) \in  L^2(\dD\mu_t)$ as
\begin{equation}\label{def_xi}
\xi_{\delta}(t,x,v) \,:=\, f(t,x,v)\,-\,\mP_{W_\delta} f(t,x,v)\,=\,\sum_{n\in\NN}\xi_{\delta,n}(t,x)\,\Psi_{n}(t,v),
\end{equation}
with
$$
  \xi_{\delta,n}(t,x) \,:=\,\hat{C}_n-\mP_{X_h}\hat{C}_n\,\mathbf{1}_{\{n<N\}}, \quad n\in\NN.
$$
Since $\eta_{\delta,n}\in X_h$, by taking $\varphi_n=\eta_{\delta,n}$ in \eqref{FVapproxcompact} and \eqref{FVexactcompact} and substrating the two equalities, we get
\begin{equation}
  \label{eqBhn}
\mB_{n}(\eta_\delta,E_{\delta},\eta_{\delta,n})\,=\, -\mB_{n}(\xi_\delta,E_{\delta},\eta_{\delta,n})\,+\,\alpha\,\sqrt{n}\int_\TT(E-E_{\delta})\,\hat{C}_{n-1}\,\eta_{\delta,n}\,\dD x.
\end{equation}
Now, the aim is to estimate the consistency error defined for any
$t\geq 0$ by
$$
  \|\mP_{W_\delta}f(t)-f_{\delta}(t)\|_{L^2(\dD\mu_t)}^2\,=\,\sum_{n=0}^{N-1}\alpha(t)\,\int_\TT|\,\eta_{\delta,n}|^2\dD
    x\,.
$$
To do this, we compute the time derivative of
$\|\mP_{W_\delta}f(t)-f_{\delta}(t)\|^2_{L^2(\dD\mu_t)}$ given by
$$
\frac{1}{2}\,\frac{\dD }{\dD t}\left[\sum_{n=0}^{N-1}\alpha(t)\,\int_\TT
 |\,\eta_{\delta,n}(t)\,|^2\dD
  x\right] \,=\,\sum_{n=0}^{N-1}\int_\TT\left[\alpha(t) \,\eta_{\delta,n}\,\partial_t\eta_{\delta,n}\,+\,\frac{1}{2}\,\alpha^\prime(t)\,|\,\eta_{\delta,n}|^2\right]\,\dD
x\,,
$$
hence using \eqref{eqBhn} together with the
  definition \eqref{defBnh} of $\mB_{n}$, we  obtain
  \begin{equation}
  \label{deb_consis}
\frac{1}{2}\,\frac{\dD }{\dD t}\left[\sum_{n=0}^{N-1}\alpha(t)\,\int_\TT
 |\,\eta_{\delta,n}(t)\,|^2\dD
  x\right] \,=\,
\mI_1 + \mI_2 + \mI_3,
\end{equation}
where
\begin{equation}
  \label{def:I}
  \left\{
    \begin{array}{l}
 \ds\mI_1 \,:=\, -\sum_{n=0}^{N-1}\alpha(t)\, \left( \sum_{j\in\mJ}\mA_{n,j}(g_n(\eta_\delta),\eta_{\delta,n})
      \,-\,\int_\TT\mS_n[\eta_\delta,E_\delta]\,\eta_{\delta,n}\dD x-\frac{1}{2}\frac{\alpha'(t)}{\alpha(t)}\int_\TT|\eta_{\delta,n}|^2\,\dD x\right),
   \\[1.1em]
 \ds \mI_2\,:=\,
      \sum_{n=0}^{N-1}\alpha^2(t)\,\sqrt{n}\int_\TT (E-E_{\delta})\,\hat{C}_{n-1}\,\eta_{\delta,n}\,\dD x,
      \\[1.1em]
 \ds \mI_3\,:=\,  -\sum_{n=0}^{N-1}\alpha(t)\,\mB_{n}(\xi_\delta,E_{\delta}, \eta_{\delta,n}).
    \end{array}\right.
  \end{equation}

Let us now estimate each of these terms separately. Throughout the
following computations,  $\mC$ will be a generic positive constant, depending on the $L^2(\dD\mu_t)$ of the exact solution $f$ and its derivatives, but independent of $\delta\,=\,(h,1/N)$, and which may be different in different places.


We proceed as in the stability analysis detailed in \cite[Proposition
3.2]{BCF2021} or in Proposition \ref{prop:stab_L2dvdx}, replacing $C_{\delta}$ by $\eta_{\delta}$. We get 
\begin{equation}
  \label{estim:I1}
\mI_1 \,\leq\,
-\frac{1}{2}\sum_{n=0}^{N-1}\sum_{j\in\hat\mJ}\nu_n\,[\eta_{\delta,n}]^2_{j-{1}/{2}}\,+\,\frac{1}{4\,\gamma}\|\eta_{\delta}(t)\|_{L^2(\dD\mu_t)}^{2}.
\end{equation}
For the second term $\mI_2$, we apply the Cauchy-Schwarz inequality
and  write
\begin{equation*}
 |\mI_2|\,\leq\, \|E-E_{\delta}\|_{L^{\infty}(\dD
  x)}\left(\sum_{n=0}^{N-1}\alpha^3(t)\,n\int_\TT |\hat{C}_{n-1}|^2\,\dD x\right)^{{1}/{2}}\left(\sum_{n=0}^{N-1}\alpha(t)\int_\TT|\eta_{\delta,n}|^2\dD x\right)^{{1}/{2}}\,.
\end{equation*}
On the one hand, to estimate $\|E-E_{\delta}\|_{L^{\infty}(\dD x)}$, we proceed as in the proof of \cite[Proposition 2.3]{BCF2021}. By \textcolor{black}{the} Sobolev and Poincaré-Wirtinger inequalities, there exists a constant $\mC>0$ such that 
\[
  \|E-E_{\delta}\|_{L^{\infty}}^2\leq
  \mC\,\|\partial_x(E-E_{\delta})\|^2_{L^2(\dD x)}.
\]
Substrating \eqref{approxPoisson} and the second equation of \eqref{vlasov}, taking $\partial_x(E-E_{\delta})$ as test function, and applying \textcolor{black}{the} Cauchy-Schwarz inequality, it yields
\begin{equation*}
 \|\partial_x(E-E_{\delta})\|_{L^2}^2\,\leq\, \|\partial_x(E-E_{\delta})\|_{L^2} \,\|\hat{C}_0-C_{\delta,0}\|_{L^2},
 \end{equation*}
 and then
$$
   \|E-E_{\delta}\|_{L^{\infty}}^2\,\leq\, \mC
                                    \,\|\hat{C}_0-C_{\delta,0}\|_{L^2}^2
   \,\leq\, \frac{\mC}{\alpha(t)}\,\|f(t)-f_{\delta}(t)\|_{L^2(\dD\mu_t)}^2\,.
 $$
On the other hand,  we remark that using \textcolor{black}{the} decomposition
\eqref{decompo_f} of $f$ and the third equality in \eqref{vPsi}, we have
 \[ \partial_v f=-\sum_{n=0}^{+\infty}\alpha\,\sqrt{n}\,\hat{C}_{n-1}\Psi_n,\]
 and then
 \[
   \sum_{n=0}^{N-1}\alpha^3(t)\,n\int_\TT |\hat{C}_{n-1}|^2\,\dD x
  \,\leq\, \sum_{n\in\NN} \alpha(t)\int_\TT
  |\alpha(t)\,\sqrt{n}\,\hat{C}_{n-1}|^2\,\dD x \,=\, \|\partial_v
  f\|_{L^2(\dD\mu_t)}^2.
\]
Gathering these results,  we obtain an estimate on $\mI_2$ as
\begin{equation*}
|\mI_2|\,\leq\, \frac{\mC}{\sqrt{\alpha}(t)}\,\|f(t)-f_{\delta}(t)\|_{L^2(\dD\mu_t)}\,\|\eta_{\delta}(t)\|_{L^2(\dD\mu_t)}.
\end{equation*}
Since the total error is the sum of the projection and consistency
errors, we get by \textcolor{black}{the} triangle inequality and the use of the first item of Theorem \ref{thm_errproj} that
$$
|\mI_2|\,\leq \, \frac{\mC}{\sqrt{\alpha}(t)}\left[
  \frac{1}{\left(\alpha^2\,N\right)^{m/2}}\,+\,h^{k+1} \,+\,
  \|\eta_{\delta}\|_{L^2(\dD\mu_t)}
\right]\,\|\eta_{\delta}\|_{L^2(\dwdx)}\,,
$$
where the constant $\mC>0$ depends on the weighted $H^m$ norm of
$f(t)$. Finally applying \textcolor{black}{the} Young inequality, we conclude that
\begin{equation}
  \label{estim:I2}
|\mI_2|\,\leq\,
\frac{\mC}{\sqrt{\alpha}(t)}\left(\frac{1}{\left(\alpha^2\,N\right)^{m}}\,+\,h^{2(k+1)}\,+\,
  \| \eta_{\delta}\|_{L^2(\dD\mu_t)}^2\right)\,.
\end{equation}
Now we turn to the last term $\mI_3$ and use the definition
\eqref{defBnh} of $\mB_{n}$ to obtain
\begin{equation*}
\mI_3\,=\,\mI_{31}\,+\,\mI_{32}\,+\,\mI_{33},
\end{equation*}
with
$$
\left\{\begin{array}{l}
\ds\mI_{31} \,:=\,
         -\sum_{n=0}^{N-1}\alpha(t)\int_\TT\left(\partial_t\xi_{\delta,n}\,-\,\frac{\alpha^\prime(t)}{\alpha(t)}\,\left(n\,\xi_{\delta,n}\,+\,\sqrt{(n-1)n}\,\xi_{\delta,n-2}\right)\right)\,\eta_{\delta,n}\dD
         x\,,
         \\[1.1em]
\ds\mI_{32}\,:=\,+\sum_{n=0}^{N-1}\alpha^2(t)\,\sqrt{n}\int_\TT E_{\delta}\,\xi_{\delta,n-1}\,\eta_{\delta,n}\,\dD x\,,
         \\[1.1em]
\ds\mI_{33}\,:=\,-\sum_{n=0}^{N-1}\alpha(t)
         \sum_{j\in\mJ}\mA_{n,j}(g_n(\xi_\delta),\eta_{\delta,n})\,.
       \end{array}\right.
     $$
     We start with $\mI_{31}$ and  remark that
$$
\sum_{n=0}^{N-1}\alpha(t)\int_\TT\left[\partial_t\xi_{\delta,n}-\frac{\alpha^\prime(t)}{\alpha(t)}\left(n\,\xi_{\delta,n}\,+\,\sqrt{(n-1)n}\,\xi_{\delta,n-2}\right)\right]^2\dD
x\,\leq\, \|\partial_t f(t)-\mP_{W_\delta}\partial_t f(t)
\|_{L^2(\dD\mu_t)}^2\,.
$$
Then, since $f(t)\in \HH^m(\dD\mu_t)$ satisfies the Vlasov equation
\eqref{vlasov}, and using the second estimate \eqref{lemme2_ineg1}
of Lemma \ref{lemme2}, we have
\begin{eqnarray*}
\|\partial_t f\|_{\HH^{m-1}(\dD\mu_t)}&\leq&
                                          \|v\,\partial_xf\|_{\HH^{m-1}(\dD\mu_t)}\,+\,\|E\|_{L^{\infty}}\,\|\partial_v f\|_{\HH^{m-1}(\dD\mu_t)}\,
  \\
&\leq & \frac{\mC}{\alpha^2(t)}\,\|f\|_{\HH^m(\dD\mu_t)}\,.
\end{eqnarray*}
Thus, applying the Cauchy-Schwarz inequality to $\mI_{31}$ and using
Theorem \ref{thm_errproj} to $\partial_t f$ with $r=m-1$, we have
\begin{equation}
  \label{estim:I31}
|\mI_{31}| \,\leq\, \frac{\mC}{\alpha^2(t)}\left(\frac{1}{(\alpha^2(t) N)^{(m-1)/2}} \,+\,h^{k+1}\right)\,\|\eta_{\delta}\|_{L^2(\dD\mu_t)}.
\end{equation}
The estimate on $\mI_{32}$ follows the same lines as the one for
$\mI_{31}$. Indeed, remarking again that 
\[
  \sum_{n=0}^{N-1}\alpha(t)\int_\TT\left|\alpha(t)\,\sqrt{n}\,\xi_{\delta,n-1}\right|^2\dD
  x \,\leq\, \|\partial_v f-\mP_{W_\delta}\partial_v
  f\|_{L^2(\dD\mu_t)}^2\,,
\]
and applying the Cauchy-Schwarz inequality to $\mI_{32}$ and Theorem
\ref{thm_errproj} to $\partial_v f$ with $r=m-1$, we get
\begin{equation*}
|\mI_{32}|\,\leq\,
\mC\,\|E_{\delta}\|_{L^\infty}\,\left(\frac{1}{(\alpha^2(t)\,N)^{(m-1)/2}}
  \,+\,h^{k+1}\right)\, \left(\alpha(t)\sum_{n=0}^{N-1}\int_\TT|\,\eta_{\delta,n}\,|^2\dD x\right)^{{1}/{2}}\,,
\end{equation*}
which can be written as 
\begin{equation}
  \label{estim:I32}
|\mI_{32}|\leq \frac{\mC}{\sqrt{\alpha}(t)}\left(\frac{1}{(\alpha^2(t)\,N)^{(m-1)/2}}\,+\,h^{k+1}\right)\,\|\eta_{\delta}(t)\|_{L^2(\dD\mu_t)}\,.
\end{equation}

Finally, we turn to the estimation of $\mI_{33}$ and split it as $\mI_{33}\,=\,\mI_{331}\,+\,\mI_{332}$, with
$$
\left\{\begin{array}{l}
\ds \mI_{331}\,:=\,
         \sum_{n=0}^{N-1}\alpha(t) \,\sum_{j\in\mJ}\int_{I_j}g_n(\xi_\delta)\,\partial_x\eta_{\delta,n}\,\dD
         x,\\[1.1em]
\ds         \mI_{332}\,:=\, -\sum_{n=0}^{N-1}\alpha(t)\sum_{j\in\mJ}\left((\hat{g}_n(\xi_\delta)\,\eta_{\delta,n}^-)_{j+{1}/{2}}-(\hat{g}_n(\xi_\delta)\,\eta^+_{\delta,n})_{j-{1}/{2}}\right)\,.
       \end{array}\right.
 $$
Using the definition \eqref{anh} of $g_n(\xi_\delta)$, we have
\begin{equation*}
\mI_{331}\,=\,\sum_{n=0}^{N-1}\sum_{j\in\mJ}\int_{I_j}\left(\,\sqrt{n}\,\xi_{\delta,n-1}\,+\, \sqrt{n+1}\,\xi_{\delta,n+1}\,\right)\partial_x\eta_{\delta,n}\dD x\,.
\end{equation*}
Since $\xi_{\delta,n}=\hat{C}_n-\mP_{X_h}\hat{C}_n$ for $n=0,\ldots,N-1$ and
$\partial_x\eta_{\delta,n}\in X_h$, we have by definition of the
projection $\mP_{X_h}$ that $\mI_{331}=0$. Now, it remains to
estimate  $\mI_{332}$. Using \textcolor{black}{the} periodic boundary conditions, we may write
it as
\begin{equation*}
\mI_{332}\,=\,\sum_{n=0}^{N-1}\alpha(t)\sum_{j\in\hat\mJ}\left(\,\hat{g}_n(\xi_\delta)\,[\eta_{\delta,n}]\right)_{j-{1}/{2}}\,.
\end{equation*}
Then, applying the Young inequality, we obtain
\begin{equation*}
|\mI_{332}|\,\leq\,
\frac{\alpha(t)}{2}\,\sum_{n=0}^{N-1}\sum_{j\in\hat\mJ}\left( \, \frac{\alpha(t)}{\nu_n}\left|\,\hat{g}_n(\xi_\delta)_{j-{1}/{2}}\,\right|^2\,+\,\frac{\nu_n}{\alpha(t)}\,[\eta_{\delta,n}]_{j-{1}/{2}}^2\,\right)\,.
\end{equation*}

The second term of the right hand side will be balanced with the
dissipation term figuring in the estimate \eqref{estim:I1} of $\mI_1$,
whereas the first term is estimated as follows. Using the definition
of the numerical flux $\hat{g}_n$  \eqref{lfflx}  and the artificial viscosity $\nu_n$, we have 
\begin{eqnarray*}
 \frac{\alpha^2(t)}{2}\,\sum_{n=0}^{N-1}\sum_{j\in\hat\mJ}\frac{1}{\nu_n}\,\left|\,\hat{g}(\xi_\delta)_{n,j-{1}/{2}}\,\right|^2
  &\leq&
         \sum_{n=0}^{N-1}\sum_{j\in\hat\mJ}\frac{1}{\underline{\nu}} \,\left\{\sqrt{n}\,\xi_{\delta,n-1} \,+\,\sqrt{n+1}\,\xi_{\delta,n+1}\right\}_{j-{1}/{2}}^2
  \\
  &+&\overline{\nu}\,\sum_{n=0}^{N-1}\sum_{j\in\hat\mJ}[\xi_{\delta,n}]_{j-{1}/{2}}^2\,.
\end{eqnarray*}
On the one hand, observing that
\[
  v\,f-\mP_{W_\delta}(v\,f)=\sum_{n\in\NN}\frac{1}{\alpha}\left(\sqrt{n}\,\xi_{\delta,n-1}+\sqrt{n+1}\,\xi_{\delta,n+1}\right)\Psi_n,
\]
we apply the second item of Theorem \ref{thm_errproj} to $g=vf$ with $r=m-1$ and the second
inequality \eqref{lemme2_ineg1} of Lemma \ref{lemme2}, 
\begin{eqnarray*}
\sum_{n=0}^{N-1}\sum_{j\in\hat\mJ}\frac{1}{\underline{\nu}}
\,\left\{\sqrt{n}\,\xi_{\delta,n-1}
  \,+\,\sqrt{n+1}\,\xi_{\delta,n+1}\right\}_{j-{1}/{2}}^2
  &\leq&  \mC\,\alpha^2\,\int_{\RR}\|v\,f-\mP_{W_\delta}(v\,f)\|_{L^2(\Gamma)}^2\,\dw
\\
  &\leq&  \mC \left(
  \frac{1}{(\alpha^2(t) N)^{m-1}}+
         h^{2k+1}\right)\,\|f(t)\|_{\HH^{m}(\dD\mu_t)}^2\,.
\end{eqnarray*}
The viscosity term is estimated in the same manner, hence  we deduce that
\begin{equation}
  \label{estim:I33}
|\mI_{33}|\leq
\mC\, \left(
  \frac{1}{\left(\alpha^2(t)\,N\,\right)^{m-1}}\,+\,h^{2k+1}\right)\,\|f(t)\|_{\HH^m(\dD\mu_t)}^2
\,+\,\frac{1}{2}\sum_{n=0}^{N-1}\sum_{j\in\hat\mJ}\nu_n \,[\eta_{n,h}]_{j-{1}/{2}}^2.
\end{equation}

Gathering \eqref{estim:I31}, \eqref{estim:I32} and \eqref{estim:I33}, we obtain
\begin{eqnarray*}
|\mI_3|&\leq
  &\,\frac{\mC}{\alpha^2(t)}\left(\frac{1}{(\alpha^2(t)\,N)^{(m-1)/2}}\,+\,h^{k+1}\right)\,\|\eta_{\delta}(t)\|_{L^2(\dD\mu_t)}
  \\
& +&
     \mC\,\left(\frac{1}{(\alpha^2(t)\,N)^{m-1}}\,+\, h^{2k+1}\right)\,+\,
     \frac{1}{2} \sum_{n=0}^{N-1}\sum_{j\in\hat\mJ}\nu_n\,[\eta_{\delta,n}(t)]_{j-{1}/{2}}^2. 
\end{eqnarray*}
Finally applying \textcolor{black}{the} Young inequality to the first term, we conclude that
\begin{equation}
  \label{estim:I3}
|\mI_3|\leq \frac{\mC}{\alpha^4(t)}\left( \frac{1}{(\alpha^2(t)\,N)^{m-1}}\,+\,h^{2k+1}\right)\,+\,\|\eta_{\delta}(t)\|_{L^2(\dD\mu_t)}^2 \,+\,\frac{1}{2}\sum_{n=0}^{N-1}\sum_{j\in\hat\mJ}\nu_n\,[\eta_{n,h}(t)]_{j-{1}/{2}}^2. 
\end{equation}
Collecting \eqref{estim:I1}, \eqref{estim:I2} and \eqref{estim:I3} in \eqref{deb_consis}, it yields
  \begin{eqnarray*}
    \frac{1}{2}\frac{\dD }{\dD
    t}\|\mP_{W_\delta}f(t)-f_{\delta}(t)\|_{L^2(\dD\mu_t)}^2 &\leq&  \frac{\mC}{\sqrt{\alpha(t)}}\,\|\mP_{W_\delta}f(t)-f_{\delta}(t)\|_{L^2(\dD\mu_t)}^2\\
 &+&\frac{\mC}{\alpha^4(t)}\left(\frac{1}{(\alpha^2(t)\,N)^{m-1}}\,+\,h^{2k+1}\right).
\end{eqnarray*}

We conclude using the Gronwall lemma that for any $t\geq 0$, we have
\begin{equation*}
\|\mP_{W_\delta}f(t)-f_{\delta}(t)\|_{L^2(\dD\mu_t)}^2\,\leq \,
 \int_0^t\frac{\mC}{\alpha^4(s)}\,\left(\frac{1}{(\alpha^2(s)\,N)^{m-1}}+h^{2k+1}\right)\,\dD
  s\,\, \exp\left(\int_0^t\frac{\mC\dD\tau}{\sqrt{\alpha}(\tau)}\right)\,.
\end{equation*}
Using the projection estimate of Theorem \ref{thm_errproj} together with
the latter inequality, and since \textcolor{black}{by Proposition \ref{prop:bornes_alpha}} the function $\alpha$ is bounded from
below and above on a finite interval of time, we have established \eqref{res:main}.

\section{Numerical simulations}
\setcounter{equation}{0}
\label{sec:5}

It is worth to mention that the Hermite/di\textcolor{black}{s}continuous Galerkin method
has already been validated in  \cite{Filbet2020,BCF2021} on classical
numerical tests. \textcolor{black}{Hence in this section, we perform complementory numerical simulations on the
Vlasov-Poisson system \eqref{vlasov} using \textcolor{black}{the} DG/Hermite Spectral
method to illustrate our theoretical result and to investigate the
\textcolor{black}{impact of the choice} of the  free parameter
$\gamma\geq 0$ which enters in the definition of the scaling function 
$\alpha$ in \eqref{eq:def_alpha}. We also refer to \cite{Tang1993}
for a discussion on the latter point.}

\subsection{Test 1: order of convergence}
We take $N$ modes for \textcolor{black}{the} Hermite spectral bases and $N_x$ cells
in space, and apply a third order Runge-Kutta scheme for the time
discretization with a small time step  $\Delta t=0.001$ in order to neglect the time
discretization error.  The initial scaling parameter $\alpha(0)$ is chosen to be $1$ and the Hou-Li filter with $2/3$ dealiasing rule \cite{hou2007computing,filtered} will be used.

\textcolor{black}{We choose the following initial condition
\beq
\label{landau}
f_0(x,v)\,=\,\left(1+\delta \cos\left(k \,x\right)\right)\f{1}{\sqrt{2\pi}}\exp\left(-\f{v^2}{2}\right)\,,
\eeq
with $\delta=0.01$, which corresponds to the Landau damping
configuration.  The background density is $\rho_0=1$, the length of
the domain in the $x$-direction is $L=4\pi$ (that is $k=\pi/6$) and the final time
is $T=0.1$. The free parameter $\gamma$ is chosen as $\gamma=1$ and
the errors are computed by comparing to a reference
solution obtained using $N_x\times N =512\times 512$ with $P_2$ piecewise polynomial
basis and Hermite polynomial in velocity. In Table \ref{tab41}, we
show the weighted $L^2$ errors and orders for $P_k$ piecewise polynomials with $k=1,2$ respectively. Due to the fact that \textcolor{black}{the} time steps are smaller than the spatial mesh size, we can observe $(k+1)$-th order of convergence for $P_k$ polynomials respectively.}

\textcolor{black}{\begin{table}[!ht] 
	\centering
	\label{tab41}
	\vspace{0.15in}
	\begin{tabular}{|l|l|l|l|l|} \hline
	\multicolumn{1}{|l|}{}&\multicolumn{2}{|c|}{$P_1$}&\multicolumn{2}{|c|}{$P_2$} \cr\hline 		
	$N_x\times N$ & $L^2(\dD\mu_t)$ error & Order & $L^2(\dD\mu_t)$ error & Order   \\ \hline
	$16\times 16$ & 5.12E-4 & --   & 1.44E-5 & --      \\ \hline
        $32\times 32$ & 1.05E-4 & 2.28   & 1.68E-6 & 3.09      \\ \hline
	$64\times 64$ & 2.31E-5 & 2.18 & 2.05E-7 & 3.04  \\ \hline
	$128\times 128$& 5.42E-6 & 2.09 & 2.48E-8 & 3.04  \\ \hline
	\end{tabular}
	\caption{Test 1 : Numerical weighted $L^2$ errors and orders for Landau damping with initial distribution \eqref{landau}, $\delta=0.01$ and $k=\pi/6$, $T=0.5$.}
      \end{table}}

\subsection{Test 2: Bump-on-the-tail}
\textcolor{black}{Now we investigate the impact  of the parameter $\gamma$ appearing in the
definition of  scaling function $\alpha$ in
\eqref{eq:def_alpha}. According to our analysis, when
$\gamma$ decreases, the function $\alpha$ decays more slowly and the
weighted  norm $\|f_\delta(t)\|_{L^2(\dD\mu_t)}$  is bounded as
\beq
 \|f_\delta(t)\|_{L^2(\dD\mu_t)} \,\leq\, \|f_{\delta}(0)\|_{L^2(\dD\mu_0)}\,e^{t/4\gamma}\,.
\label{riri}
 \eeq
For practical computation, we expect that the scaling function
$\alpha$ follows the variation of the distribution function in the
velocity space, hence it is crucial to have a good understanding of
the impact of this free parameter. For these reasons we
present a numerical example on the bump-on-the-tail \cite[Section 5.2]{BCF2021} where the
distribution function strongly varies in $v$. We consider the initial distribution as
\begin{align}
\label{bot}
f(0,x,v)=f_{b}(v)(1+\kappa\cos(k\,n\,x))\,,
\end{align} 
where the bump-on-tail distribution is
\begin{align}
f_{b}(v)=\frac{n_p}{\sqrt{\pi}v_{p}}e^{-v^2/v^2_{p}}+\frac{n_b}{\sqrt{\pi}v_{b}}e^{-(v-v_{d})^2/v_{b}^2}\,.
\end{align}
We choose a strong perturbation with $\kappa=0.04$, $n=3$ and
$k=2\pi/L$ with $L=62$ and the other parameters are set to be
$n_p=0.9$, $n_b=0.1$, $v_{d}=4.5$, $v_{p}=\sqrt{2}$,
$v_{b}=\sqrt{2}/2$. The computational domain is $[0,L]\times[-8, 8]$.
These settings have been used in \cite{shoucri1974} and \cite[Section
4.3]{Filbet2020}. For this case, we take the initial scaling function
to be   $\alpha_0 =5/7$. We take $N_x\times N=64\times 128$.
\\
For the Vlasov-Poisson system, we know that the total energy and the $L^2$
norm of $f$ are exactly conserved, but this property is no more true
at the discrete level.  In Figure \ref{fig:1} $(a)$-$(b)$, we present the time
evolution of these quantities for different values of $\gamma$
corresponding to the  definition \eqref{eq:def_alpha}. On the one
hand, the amplitude of the variations of the total energy are of order
$10^{-6}$ for the different values of $\gamma$ and are even smaller
when $\gamma$ is small. On the other hand, the $L^2$ norm of $f$
oscillates around its initial value, but again the impact of the
parameter $\gamma$ is negligible (see $(b)$). We also present
the time evolution of the potential and kinetic energy in  Figure
\ref{fig:1} $(c)$-$(d)$ for different values of $\gamma$.  From these
plots, we can observe that the impact of this free parameter is
limited and does not affect the accuracy of the method.}

\textcolor{black}{\begin{figure}
	\centering
        \begin{tabular}{cc}
          \includegraphics[width=3.2in,clip]{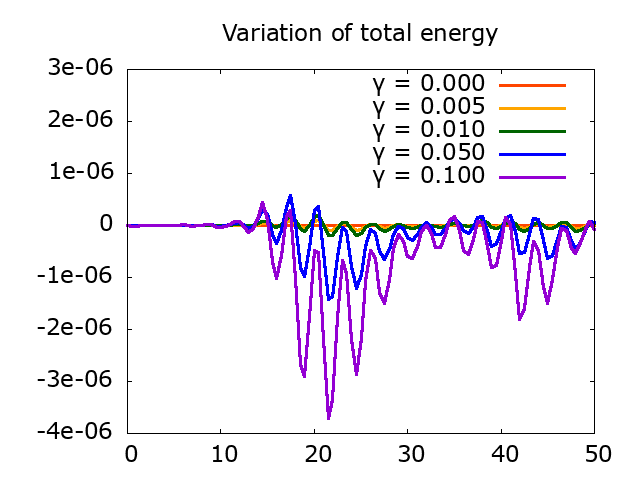}&
           \includegraphics[width=3.2in,clip]{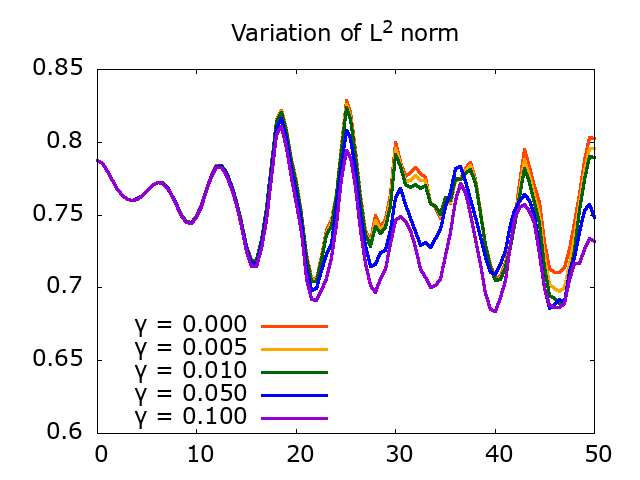} \\
                                                         (a)&(b) \\
         \includegraphics[width=3.2in,clip]{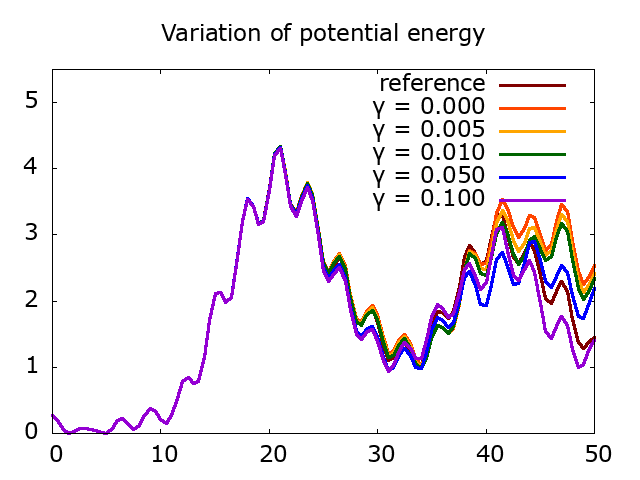}&
	\includegraphics[width=3.2in,clip]{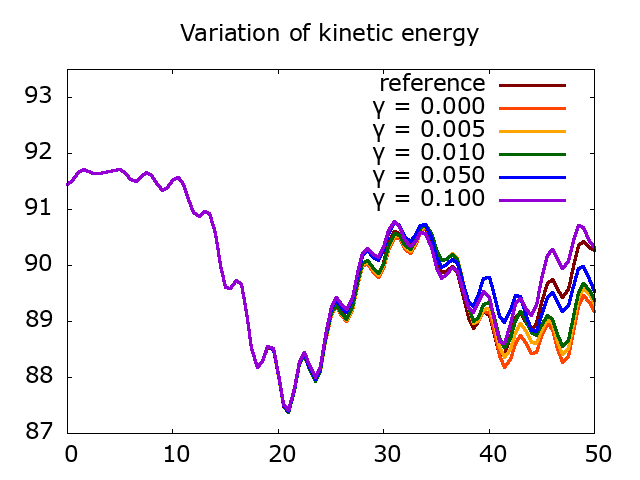}\\
        (c)&(d)
        \end{tabular}
	\caption{{\bf Test 2:}   time evolution of (a) the variation of the total energy,
          (b) the  standard $L^2$ norm of $f$, (c) the potential energy and (d) the kinetic energy  for $\alpha$ given by 
          \eqref{eq:def_alpha}, with several values of $\gamma=0,...\,, \, 10^{-1}$. }
	\label{fig:1}
\end{figure}
}

\textcolor{black}{Finally, in Figure \ref{fig:2}, we present the time
evolution of $\alpha$ and the corresponding weighted $L^2$ norm for
different values of $\gamma$. Since the initial data $\alpha(0)$ is
the same for our simulations, we know that when $\gamma\leq \gamma'$,
we have $\alpha_{\gamma'}(t) \leq \alpha_{\gamma}(t)$ (since
$\alpha_{\gamma'}$  decays 
faster than $\alpha_{\gamma}$) and
then
$$
\|f(t) \|_{L^2(\dD\mu_t')} \leq \|f(t) \|_{L^2(\dD\mu_t)}, 
$$
where $\mu_t'$ represents the measure associated to
$\alpha_{\gamma'}(t)$.  We also notice that when $\gamma$ increases,
the weighted $L^2$ norm $\|f(t) \|_{L^2(\dD\mu_t)},$ grows slowly,
which is consistent with our estimate  \eqref{riri}.  Moreover, for a fixed $\gamma$,  the time evolution of $\|f(t)
\|_{L^2(\dD\mu_t)}$ first follows the growth of the potential energy which
is almost exponentially fast, but when nonlinear effects dominate, it
starts to oscillate and is stabilized. This last numerical result
illustrates that the estimate \eqref{riri} is certainly not optimal
for large time...}

\textcolor{black}{\begin{figure}
	\centering
        \begin{tabular}{cc}
          \includegraphics[width=3.2in,clip]{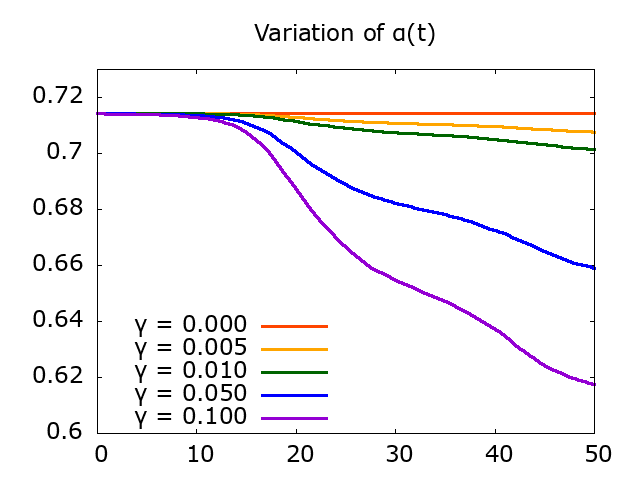}&
          \includegraphics[width=3.2in,clip]{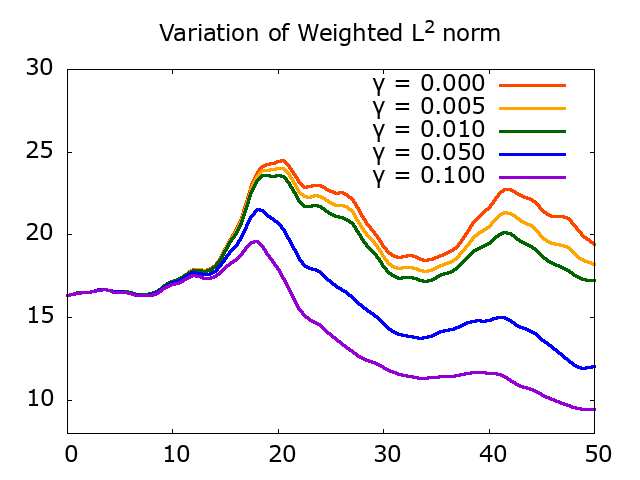} \\
        (a)&(b)
        \end{tabular}
	\caption{{\bf Test 2:}   time evolution of (a) the scaling
          function $\alpha$  and (b) the corresponding  weighted $L^2$ norm
          of $f$, for $\alpha$ given by \eqref{eq:def_alpha}, with
          several values of $\gamma=0,...\,, \, 10^{-1}$.}
	\label{fig:2}
\end{figure}
}

\section{Conclusion \& Perspectives}
\setcounter{equation}{0}
\label{sec:6}

In this article we investigate the convergence
analysis of  a spectral Hermite discretization of the Vlasov-Poisson
system with a time-dependent scaling factor allowing to prove \textcolor{black}{the}
stability and convergence of the numerical solution in an appropriate
functionnal framework. The control of this scaling factor, and more
precisely a positive lower bound, is crucial to ensure completely the
stability of the method and our proof follows carefully this
dependency. Our analysis is limited to the one dimensional case for
which the control of the electric field is straightforward. In a
future work we would like to adapt our approach to the
multi-dimensional case following the ideas in \cite{Ayuso2012} for the
control of the $L^\infty$ norm of the electric field in two and three dimensions.

%
\section*{Acknowledgement}
Both authors want to thank the two referees for their motivating
comments to improve our analysis.

Marianne Bessemoulin-Chatard is partially funded by the Centre Henri
Lebesgue (ANR-11-LABX-0020-01) and ANR Project MoHyCon
(ANR-17-CE40-0027-01). Francis Filbet  is partially funded by the ANR
Project Muffin (ANR-19-CE46-0004) and by the EUROfusion Consortium and
has received funding from the Euratom research and training programme 2019-2020 under grant agreement No 633053. The views and opinions expressed herein do not necessarily reflect those of the European Commission.

\bibliographystyle{abbrv}
\bibliography{refer_new}

\begin{thebibliography}{10}

\bibitem{Armstrong1967}
T.~P. Armstrong.
\newblock {Numerical studies of the nonlinear Vlasov equation}.
\newblock {\em The Physics of Fluids}, 10(6):1269--1280, 1967.

\bibitem{Ayuso2011}
B.~Ayuso, J.~A. Carrillo, and C.-W. Shu.
\newblock Discontinuous {G}alerkin methods for the one-dimensional
  {V}lasov-{P}oisson system.
\newblock {\em Kinet. Relat. Models}, 4(4):955--989, 2011.

\bibitem{Ayuso2012}
B.~Ayuso, J.~A. Carrillo, and C.-W. Shu.
\newblock {Discontinuous Galerkin methods for the multi-dimensional
  Vlasov--Poisson problem}.
\newblock {\em Mathematical Models and Methods in Applied Sciences},
  22(12):1250042, 2012.

\bibitem{besse}
N.~Besse.
\newblock Convergence of a high-order semi-{L}agrangian scheme with propagation
  of gradients for the one-dimensional {V}lasov-{P}oisson system.
\newblock {\em SIAM J. Numer. Anal.}, 46(2):639--670, 2008.

\bibitem{BCF2021}
M.~Bessemoulin-Chatard and F.~Filbet.
\newblock {On the stability of conservative discontinuous Galerkin/Hermite
  spectral methods for the Vlasov-Poisson system}.
\newblock {\em Journal of Computational Physics}, 451:110881, 2022.

\bibitem{Birdsall1985}
C.~K. Birdsall and A.~B. Langdon.
\newblock {\em Plasma physics via computer simulation}.
\newblock McGraw-Hill, New York, 1985.

\bibitem{Camporeale2016}
E.~Camporeale, G.~Delzanno, B.~Bergen, and J.~Moulton.
\newblock {On the velocity space discretization for the Vlasov–Poisson
  system: Comparison between implicit Hermite spectral and Particle-in-Cell
  methods}.
\newblock {\em Computer Physics Communications}, 198:47--58, 2016.

\bibitem{charles}
F.~Charles, B.~Despr\'{e}s, and M.~Mehrenberger.
\newblock Enhanced convergence estimates for semi-{L}agrangian schemes:
  application to the {V}lasov-{P}oisson equation.
\newblock {\em SIAM J. Numer. Anal.}, 51(2):840--863, 2013.

\bibitem{Cheng2013}
Y.~Cheng, I.~M. Gamba, and P.~J. Morrison.
\newblock {Study of conservation and recurrence of Runge--Kutta discontinuous
  Galerkin schemes for Vlasov--Poisson systems}.
\newblock {\em Journal of Scientific Computing}, 56(2):319--349, 2013.

\bibitem{filtered}
Y.~Di, Y.~Fan, Z.~Kou, R.~Li, and Y.~Wang.
\newblock {Filtered Hyperbolic Moment Method for the Vlasov Equation}.
\newblock {\em Journal of Scientific Computing}, 79(2):969--991, 2019.

\bibitem{duclous}
R.~Duclous, B.~Dubroca, F.~Filbet, and V.~Tikhonchuk.
\newblock {High order resolution of the Maxwell--Fokker--Planck--Landau model
  intended for ICF applications}.
\newblock {\em Journal of Computational Physics}, 228(14):5072--5100, 2009.

\bibitem{engelmann1963}
F.~Engelmann, M.~Feix, E.~Minardi, and J.~Oxenius.
\newblock {Nonlinear effects from Vlasov's equation}.
\newblock {\em The Physics of Fluids}, 6(2):266--275, 1963.

\bibitem{Filbet2001}
F.~Filbet.
\newblock Convergence of a finite volume scheme for the {V}lasov-{P}oisson
  system.
\newblock {\em SIAM J. Numer. Anal.}, 39(4):1146--1169, 2001.

\bibitem{Filbet2003}
F.~Filbet and E.~Sonnendr\"{u}cker.
\newblock Comparison of {E}ulerian {V}lasov solvers.
\newblock {\em Comput. Phys. Comm.}, 150(3):247--266, 2003.

\bibitem{filbet2006}
F.~Filbet and E.~Sonnendr{\"u}cker.
\newblock Modeling and numerical simulation of space charge dominated beams in
  the paraxial approximation.
\newblock {\em Mathematical Models and Methods in Applied Sciences},
  16(05):763--791, 2006.

\bibitem{Filbet2020}
F.~Filbet and T.~Xiong.
\newblock {Conservative Discontinuous Galerkin/Hermite Spectral Method for the
  Vlasov–Poisson System}.
\newblock {\em Commun. Appl. Math. Comput.}, 2020.

\bibitem{Fok2001}
J.~C.~M. Fok, B.~Guo, and T.~Tang.
\newblock {Combined Hermite spectral-finite difference method for the
  Fokker-Planck equation}.
\newblock {\em Mathematics of computation}, 71(240):1497--1528, 2001.

\bibitem{Heath2012}
R.~E. Heath, I.~M. Gamba, P.~J. Morrison, and C.~Michler.
\newblock {A discontinuous Galerkin method for the Vlasov--Poisson system}.
\newblock {\em Journal of Computational Physics}, 231(4):1140--1174, 2012.

\bibitem{Holloway1996}
J.~P. Holloway.
\newblock {Spectral velocity discretizations for the Vlasov-Maxwell equations}.
\newblock {\em Transport theory and statistical physics}, 25(1):1--32, 1996.

\bibitem{hou2007computing}
T.~Y. Hou and R.~Li.
\newblock {Computing nearly singular solutions using peseudo-spectral methods}.
\newblock {\em Journal of Computational Physics}, 226:379--397, 2007.

\bibitem{johnson1986}
C.~Johnson and J.~Pitk{\"a}ranta.
\newblock {An analysis of the discontinuous Galerkin method for a scalar
  hyperbolic equation}.
\newblock {\em Mathematics of computation}, 46(173):1--26, 1986.

\bibitem{Joyce1971}
G.~Joyce, G.~Knorr, and H.~K. Meier.
\newblock {Numerical integration methods of the Vlasov equation}.
\newblock {\em Journal of Computational Physics}, 8(1):53--63, 1971.

\bibitem{Klimas1994}
A.~J. Klimas and W.~M. Farrell.
\newblock A splitting algorithm for {V}lasov simulation with filamentation
  filtration.
\newblock {\em J. Comput. Phys.}, 110(1):150--163, 1994.

\bibitem{Kormann2021}
K.~Kormann and A.~Yurova.
\newblock {A generalized Fourier--Hermite method for the Vlasov--Poisson
  system}.
\newblock {\em BIT Numerical Mathematics}, pages 1--29, 2021.

\bibitem{LeBourdiec2006}
S.~Le~Bourdiec, F.~De~Vuyst, and L.~Jacquet.
\newblock {Numerical solution of the Vlasov--Poisson system using generalized
  Hermite functions}.
\newblock {\em Computer physics communications}, 175(8):528--544, 2006.

\bibitem{LP}
P.-L. Lions and B.~Perthame.
\newblock Propagation of moments and regularity for the {$3$}-dimensional
  {V}lasov-{P}oisson system.
\newblock {\em Invent. Math.}, 105(2):415--430, 1991.

\bibitem{Ma2005}
H.~Ma, W.~Sun, and T.~Tang.
\newblock Hermite spectral methods with a time-dependent scaling for parabolic
  equations in unbounded domains.
\newblock {\em SIAM journal on numerical analysis}, 43(1):58--75, 2005.

\bibitem{Ma2007}
H.~Ma and T.~Zhao.
\newblock {A stabilized Hermite spectral method for second-order differential
  equations in unbounded domains}.
\newblock {\em Numerical Methods for Partial Differential Equations: An
  International Journal}, 23(5):968--983, 2007.

\bibitem{manzini2016}
G.~Manzini, G.~L. Delzanno, J.~Vencels, and S.~Markidis.
\newblock {A Legendre--Fourier spectral method with exact conservation laws for
  the Vlasov--Poisson system}.
\newblock {\em Journal of Computational Physics}, 317:82--107, 2016.

\bibitem{Manzini2017}
G.~Manzini, D.~Funaro, and G.~L. Delzanno.
\newblock {Convergence of Spectral Discretizations of the Vlasov--Poisson
  System}.
\newblock {\em SIAM Journal on Numerical Analysis}, 55(5):2312--2335, 2017.

\bibitem{Pfaf}
K.~Pfaffelmoser.
\newblock Global classical solutions of the {V}lasov-{P}oisson system in three
  dimensions for general initial data.
\newblock {\em J. Differential Equations}, 95(2):281--303, 1992.

\bibitem{Schumer1998}
J.~W. Schumer and J.~P. Holloway.
\newblock {Vlasov simulations using velocity-scaled Hermite representations}.
\newblock {\em Journal of Computational Physics}, 144(2):626--661, 1998.

\bibitem{shoucri1974}
M.~Shoucri and G.~Knorr.
\newblock {Numerical integration of the Vlasov equation}.
\newblock {\em Journal of Computational Physics}, 14(1):84--92, 1974.

\bibitem{holloway2}
J.~W. Shumer and J.~P. Holloway.
\newblock {Vlasov simulations using velocity-scaled Hermite representations}.
\newblock {\em Journal of Computational Physics}, 144(2):626--661, 1998.

\bibitem{sonnen}
E.~Sonnendr{\"u}cker, F.~Filbet, A.~Friedman, E.~Oudet, and J.-L. Vay.
\newblock Vlasov simulations of beams with a moving grid.
\newblock {\em Computer Physics Communications}, 164(1-3):390--395, 2004.

\bibitem{Sonnendrucker1999}
E.~Sonnendr{\"u}cker, J.~Roche, P.~Bertrand, and A.~Ghizzo.
\newblock {The semi-Lagrangian method for the numerical resolution of the
  Vlasov equation}.
\newblock {\em Journal of Computational Physics}, 149(2):201--220, 1999.

\bibitem{Tang1993}
T.~Tang.
\newblock {The Hermite spectral method for Gaussian-type functions}.
\newblock {\em SIAM journal on scientific computing}, 14(3):594--606, 1993.

\bibitem{Ukai}
S.~Ukai and T.~Okabe.
\newblock On classical solutions in the large in time of two-dimensional
  {V}lasov's equation.
\newblock {\em Osaka Math. J.}, 15(2):245--261, 1978.

\bibitem{chang}
C.~Yang and M.~Mehrenberger.
\newblock Highly accurate monotonicity-preserving semi-{L}agrangian scheme for
  {V}lasov-{P}oisson simulations.
\newblock {\em J. Comput. Phys.}, 446:Paper No. 110632, 33, 2021.

\bibitem{Zaki1988a}
S.~Zaki, L.~Gardner, and T.~Boyd.
\newblock {A finite element code for the simulation of one-dimensional Vlasov
  plasmas. I. Theory}.
\newblock {\em Journal of Computational Physics}, 79(1):184--199, 1988.

\bibitem{Zaki1988b}
S.~I. Zaki, T.~Boyd, and L.~Gardner.
\newblock {A finite element code for the simulation of one-dimensional Vlasov
  plasmas. II. Applications}.
\newblock {\em Journal of Computational Physics}, 79(1):200--208, 1988.

\end{thebibliography}
\end{document}